\documentclass[a4paper,reqno]{amsart}

\usepackage[maxbibnames=9]{biblatex}

\usepackage{color}
\usepackage{graphicx}
\usepackage[normalem]{ulem}
\usepackage{morefloats}
\usepackage{hyperref}
\numberwithin{equation}{section}
\hypersetup{
    bookmarks=true,         
    unicode=false,          
    pdftoolbar=true,        
    pdfmenubar=true,        
    pdffitwindow=false,     
    pdfstartview={FitH},    
    pdftitle={My title},    
    pdfauthor={Author},     
    pdfsubject={Subject},   
    pdfcreator={Creator},   
    pdfproducer={Producer}, 
    pdfkeywords={keyword1, key2, key3}, 
    pdfnewwindow=true,      
    colorlinks=false,       
    linkcolor=green,          
    citecolor=green,        
    filecolor=green,      
    urlcolor=green,           
    urlbordercolor={1 1 1}  
}

\usepackage{amsmath, amsfonts, amsthm, amssymb}
\usepackage{fancyhdr}

\usepackage{biblatex}

\usepackage{mathrsfs}
\usepackage{blkarray}
\usepackage{amsthm}
\usepackage{mathtools}
\usepackage[title]{appendix}
\newtheorem{theorem}{Theorem}[section]
\newtheorem*{theorem*}{Theorem}
\newtheorem{lemma}[theorem]{Lemma}

\newcommand{\DD}{\mathbb{D}}

\newcommand{\RR}{\Rightarrow}

\newtheorem{proposition}[theorem]{Proposition}

\newtheorem{corollary}[theorem]{Corollary}
\theoremstyle{definition}
\newtheorem{definition}[theorem]{Definition}

\newtheorem{remark}[theorem]{Remark}

\theoremstyle{definition}
\newtheorem{exmp}[theorem]{Example} 
\theoremstyle{definition}
\newtheorem{question}[theorem]{Question}

\begin{document}

\title[\resizebox{3.9in}{!}{  \large INTERPOLATING SEQUENCES FOR PAIRS OF SPACES}]{\large INTERPOLATING SEQUENCES FOR PAIRS OF SPACES}

\author[\resizebox{1.35in}{!}{\small GEORGIOS TSIKALAS}]{\small GEORGIOS TSIKALAS}
\thanks{Partially supported by National Science Foundation Grant DMS 2054199}

\address{DEPARTMENT OF MATHEMATICS AND STATISTICS, WASHINGTON UNIVERSITY IN ST. LOUIS, ST. LOUIS, MO, 63136}
\email{gtsikalas@wustl.edu} 
\small
\begin{abstract}
    \small
    We characterize interpolating sequences for pairs of reproducing kernels $(s, \ell)$, where $s$ is a complete Pick factor of $\ell.$ This answers a question of Aleman, Hartz, M\raise.5ex\hbox{c}Carthy and Richter.

\end{abstract}
\maketitle

 \section{INTRODUCTION} 
\small
\subsection{The single kernel setting}
\large 
Let $\mathcal{H}$ denote a reproducing kernel Hilbert space on a nonempty set $X.$  Let $\text{Mult}(\mathcal{H})$ denote the multiplier algebra of $\mathcal{H},$ that is the set of all functions $\phi$ on $X$ that multiply $\mathcal{H}$ into itself. A sequence $\{\lambda_i\}\subset X$ is called an \textit{interpolating sequence for} $\text{Mult}(\mathcal{H})$ ((IM) for short) if, whenever $\{w_i\}\subset\ell^{\infty},$ there exists a multiplier $\phi$ such that $\phi(\lambda_i)=w_i$ for all $i$. Consider also the following \textit{weighted restriction operator} associated to $\{\lambda_i\}\subset X$
$$T: f\mapsto \bigg( \frac{f(\lambda_i)}{||k_{\lambda_i}||}\bigg) ,$$
which maps $\mathcal{H}$ into the space of all complex sequences. 
$\{\lambda_i\}$ is called an \textit{interpolating sequence for} $\mathcal{H}$ ((IH) for short) if $T(\mathcal{H})=\ell^2.$
\par In general, the set of all $\text{Mult}(\mathcal{H})$-interpolating sequences will be a strict subset of the set of all $\mathcal{H}$-interpolating sequences. However, these two classes turn out to coincide in many well-studied reproducing kernel Hilbert spaces. In particular, a class of spaces that share this property is the class of all \textit{complete Pick spaces}. A result of Agler-M\raise.5ex\hbox{c}Carthy \cite{CNPkernels} says that a reproducing kernel $s$ is an \textit{irreducible complete Pick kernel} (see Section \ref{2} for details) if it has the form 
\begin{equation}  s_w(z)=\frac{f(z)\overline{f(w)}}{1-\langle b(z), b(w) \rangle_{\mathcal{K}}}, \label{1.1} \end{equation} where $f: X\to\mathbb{C}$ is non-vanishing and $b$ is a function from X into the open
unit ball of an auxiliary Hilbert space $\mathcal{K}.$
 If, in addition, f is identically $1$ and
there is a point $w_0\in X$ such that $b(w_0)=0$, we will call (\ref{1.1}) a \textit{normalized
complete Pick kernel}. This is a rather standard assumption that is convenient but not essential for our proofs. The Hardy and Dirichlet spaces of the unit disk $\mathbb{D}$ are examples of spaces with normalized complete Pick kernels. We provide further examples and background on these kernels in Section \ref{2}.
\par Interpolating sequences are often characterized by appropriate separation and Carleson
measure conditions. If $\mathcal{H}_k$ is a reproducing kernel Hilbert space with kernel $k$, then 
$$d_k(z, w)=\sqrt{1-\frac{|\langle k_z, k_w\rangle|^2}{||k_z||^2||k_w||^2}}, \hspace{0.2 cm} z,w\in X,$$
defines a pseudometric on $X$ (see \cite[Lemma 9.9]{Pick}). Actually, $d_k$ is a metric on $X$ if and only if no two kernel functions $k_z$, $k_w$ (with $z\neq w$) are linearly dependent. In general, not much is known about $d_k$ and many natural questions remain open (see \cite{DistanceF}). In the setting of the Hardy space, $d_k$ is precisely the pseudohyperbolic metric on the unit disk. The
sequence $\{\lambda_i\}\subset X$ is said to be \textit{weakly separated by} $k$
if there exists $\epsilon>0$ such that 
\begin{equation} d_k(\lambda_i, \lambda_j)\ge \epsilon, \hspace{0.2 cm} \text{ for all } i\neq j. \tag{WS} \end{equation}

\par We also say that $\{\lambda_i\}$ satisfies the \textit{Carleson measure condition for $\mathcal{H}_k$} if there
exists $C>0$ such that
\begin{equation}
   \sum_{i=1}^{\infty}\frac{|f(\lambda_i)|^2}{||k_{\lambda_i}||^2} \le C ||f||^2_{\mathcal{H}_k}, \hspace{0.3 cm} \forall f\in\mathcal{H}_k.
    \tag{CM}
\end{equation}
Letting $\delta_{\lambda}$ denote the unit point mass at $\lambda\in X$, we see that (CM) is equivalent to $\mu=\sum_{i=1}^{\infty}\frac{1}{||k_{\lambda_i}||^2}\delta_{\lambda_i} $ being a Carleson measure for the Hilbert space $\mathcal{H}_k.$ \par 
 \par
Carleson \cite{Carleson} and Shapiro-Shields \cite{ShapiroShields} proved that (IM), (IH) and (CM)+\\(WS)  all coincide in the setting of the Hardy space on $\mathbb{D}$. Bishop \cite{Bishop} and Marshall-Sundberg \cite{MarshallSundberg} showed that this is still the case if the Hardy space is replaced by the Dirichlet space on $\mathbb{D}.$
\par As already stated, (IH) and (IM) continue to be equivalent in any complete Pick space. Also, it is not hard to see that the implication (IH)$\Rightarrow$ (CM)+(WS) is valid in every reproducing kernel Hilbert space. The question whether the converse always holds true in a complete Pick space (first formulated by Agler-M\raise.5ex\hbox{c}Carthy in \cite[Question 9.57]{Pick} and by Seip in \cite[Conjecture 1, p. 33]{Seip}) remained open for more than ten years.  It was finally given an affirmative answer by Aleman, Hartz, M\raise.5ex\hbox{c}Carthy and Richter in the breakthrough paper \cite{InterpolatingPick}, as a consequence of the positive solution of the Kadison-Singer problem \cite{Kadison}. An alternative proof, using the \textit{column-row property} for complete Pick spaces, can be found in \cite{Column-Row}. See also \cite{Boe} for partial progress regarding this problem prior to \cite{InterpolatingPick}.

\small
\subsection{Pairs of kernels}
\large

In this work, we will be concerned with the concept of interpolating sequences for multipliers
between spaces, which we now define. 
\par 
Let  $k, \ell$ be two reproducing kernels on a set $X$ such that $k_z, \ell_z\neq 0 $ for all
$z\in X$. We will denote the corresponding reproducing kernel Hilbert spaces
by $\mathcal{H}_k$ and $\mathcal{H}_{\ell}$. If $\phi\in\text{Mult}(\mathcal{H}_k, \mathcal{H}_{\ell})$, then $\phi\cdot k_z\in \mathcal{H}_{\ell}$ and so the function $\phi$ satisfies a growth estimate:
\begin{equation} \label{17} |\phi(z)|=\frac{|\phi(z)k_z(z)|}{||k_z||^2}=\frac{\langle \phi k_z, \ell_z \rangle}{||k_z||^2}\le ||\phi||_{\text{Mult}(\mathcal{H}_k, \mathcal{H}_{\ell})}\frac{||\ell_z||}{||k_z||}, \hspace{0.3 cm} \forall z\in X,  \end{equation}
where $|\phi||_{\text{Mult}(\mathcal{H}_k, \mathcal{H}_{\ell})}$ denotes the norm of the multiplication operator $M_{\phi}: \mathcal{H}_{k}\to \mathcal{H}_{\ell}$.  A sequence $\{\lambda_i\}\subset X$ will be called an \textit{interpolating sequence for $\text{Mult}(\mathcal{H}_{k}, \mathcal{H}_{\ell})$} if, whenever $\{w_i\}\subset\ell^{\infty},$ there exists a multiplier $\phi\in\text{Mult}(\mathcal{H}_{k}, \mathcal{H}_{\ell})$ such that $\phi(\lambda_i)=w_i\frac{||\ell_{\lambda_i}||}{||k_{\lambda_i}||}$ for all $i$.
\par 
Aleman, Hartz, M\raise.5ex\hbox{c}Carthy and Richter investigated interpolating sequences for pairs of spaces in \cite{InterpolatingPick}. For an arbitrary pair $(k, \ell)$, it can be shown that 
$\text{Mult}(\mathcal{H}_{k}, \mathcal{H}_{\ell})$-interpolating sequences satisfy 
the Carleson measure condition (CM) for $\mathcal{H}_k$ and are weakly separated by $\ell$. One does not expect these two conditions to also be sufficient in general. But what if, in addition, we assume  $k$ to be a \textit{complete Pick factor} of $\ell$? This is the case, for example, whenever $\mathcal{H}_k$ is the Hardy space on $\mathbb{D}$ and the operator $M_z$ of multiplication by the coordinate function defines a contraction operator on $\mathcal{H}_{\ell}.$
\begin{question}[Aleman, Hartz, M\raise.5ex\hbox{c}Carthy and Richter \cite{InterpolatingPick}] \label{4}
\textit{Let $s$ be a normalized complete Pick kernel on $X$ and let $\ell=gs$, where $g$ is a kernel on $X.$ Is it true that a sequence $\{\lambda_i\}\subset X$
is interpolating for $\text{Mult}(\mathcal{H}_{s}, \mathcal{H}_{\ell})$ if and only if it satisfies the Carleson measure condition for $\mathcal{H}_s$ and is weakly separated by $\ell$?}
\end{question}
Aleman, Hartz, M\raise.5ex\hbox{c}Carthy and Richter  were able to give a positive answer \cite[Theorem 1.3]{InterpolatingPick} to Question \ref{4} under the extra assumption that $\ell$ is a power\footnote{Notice that, by \cite[Remark 8.10]{Pick} and the Schur product theorem, the expression $s^t_w(z)$ defines a reproducing kernel whenever $s$ is a
normalized complete Pick kernel and $t>0$.} of a complete Pick kernel. 
\begin{theorem}[Aleman, Hartz, M\raise.5ex\hbox{c}Carthy and Richter \cite{InterpolatingPick}] \label{400}
Let $s_1, s_2$ be normalized complete Pick kernels on $X$ such that $s_2/s_1$ is
positive semi-definite, and let $t\ge 1$. Then, a sequence is interpolating for $\text{Mult}(\mathcal{H}_{s_1}, \mathcal{H}_{s^t_2})$ if and only if it satisfies the
Carleson measure condition for $\mathcal{H}_{s_1}$ and is weakly separated by $s^t_2$ (equivalently, by $s_2$).
\end{theorem}
Note that, for $s_1=s_2=s$ and $t=1$, their result recovers the characterization of  $\text{Mult}(\mathcal{H}_{s})$-interpolating sequences in the setting of the complete Pick space $\mathcal{H}_s.$ 
\par

\par

\small
\subsection{Main results}
\large

\par

In Section \ref{5}, we provide a complete characterization of $\text{Mult}(\mathcal{H}_{s}, \mathcal{H}_{\ell})$-interpolating sequences, thus extending Theorem \ref{400}. Surprisingly, the conditions of Question \ref{4} turn out not to be sufficient, in general, for $\text{Mult}(\mathcal{H}_{s}, \mathcal{H}_{\ell})$-interpolation. In particular, a stronger notion of weak separation is required.
\begin{definition} \label{25}
Suppose $k$ is a reproducing kernel on a nonempty set $X$ and $\{\lambda_i\}\subset X$. For any $n\ge 2$, we say that $\{\lambda_i\}$ is \textit{$n$-weakly separated by $k$} if there exists $\epsilon>0$ such that for every $n$-point subset $\{\mu_1, \dots, \mu_n\}\subset\{\lambda_i\}$ we have
$$\text{dist}\bigg(\frac{k_{\mu_1}}{||k_{\mu_1}||}, \text{ span}\bigg\{\frac{k_{\mu_2}}{||k_{\mu_2}||}, \dots, \frac{k_{\mu_n}}{||k_{\mu_n}||}\bigg\} \bigg)\ge \epsilon. $$\end{definition}
Notice that $2$-weak separation by $k$ coincides with weak separation by $k$. \par 
We can now state our first result.
\begin{theorem} \label{7}
Suppose $s$ is a normalized complete Pick kernel and $\ell=gs$ for some (positive semi-definite) kernel $g$. Then, a sequence $\{\lambda_i\}\subset X$
is interpolating for $\text{Mult}(\mathcal{H}_{s}, \mathcal{H}_{\ell})$ if and only if it satisfies the Carleson measure condition for $\mathcal{H}_s$ and is $n$-weakly separated by $\ell$, for every $n\ge 2.$
\end{theorem}
Passing to $n$-weak separation is a necessity and not merely an artifact of the proof of Theorem \ref{7}, as the following result shows.

\begin{theorem} \label{8}
There exists a kernel $\ell$ with a normalized complete Pick factor $s$ and the following property: \\
For every $n\ge 2$, there exists a sequence $\{\lambda_i\}\subset X$ that satisfies the Carleson measure condition for $\mathcal{H}_s$ and is $n$-weakly separated, but not $(n+1)$-weakly separated by $\ell$ (and hence, not $\text{Mult}(\mathcal{H}_{s}, \mathcal{H}_{\ell})$-interpolating). 
\end{theorem}
Thus, the conditions stated in Question \ref{4} are not, in general, sufficient. A natural line of inquiry then emerges: which conditions do we need to impose on a pair $(s, \ell)$ for Question \ref{4} to have a positive answer?  We investigate this in Section \ref{6}. In particular, Theorem \ref{89} tells us that, at least for ``reasonable" pairs $(s, \ell)$, the issue lies solely with the possible existence of weakly separated sequences that are not $n$-weakly separated by $\ell$ (for some $n\ge 3$). In other words, the only obstruction to Question \ref{4} having a positive answer is that $\ell$ might not possess the following (rather peculiar) property: for any fixed $n\ge 2$, a kernel $\hat{\ell}_z$ can be ``close" to the span of $n$ other kernels $\hat{\ell}_{w_1}, \hat{\ell}_{w_2}, \dots, \hat{\ell}_{w_n}$ if and only if it is ``close" to one of them. This implies, perhaps surprisingly, that the answer to Question \ref{4} is a matter that depends entirely (at least for pairs satisfying the hypotheses of Theorem \ref{89}) on the kernel $\ell$; the specific nature of the complete Pick factor $s$ turns out to be irrelevant here. Kernels for which weak separation of a sequence is always equivalent to $n$-weak separation (for \textit{every} $n$) will be said to have the \textit{automatic separation property} (also called \textit{AS property} for short). 
\par 
 The question then becomes: which kernels have the automatic separation property? This is explored in Section \ref{301}. A first class of examples is furnished by kernels satisfying a stronger property, the  \textit{multiplier separation property}. These are kernels $\ell$ such that weak separation by $\ell$ is always equivalent to weak separation\footnote{$\{\lambda_i\}$ is said to be weakly separated by $\text{Mult}(\mathcal{H}_{\ell})$ if there exists $\epsilon>0$ such that for any two points $\lambda_i\neq \lambda_j,$ we can find $\phi_{ij}\in\text{Mult}(\mathcal{H}_{\ell})$ of norm at most $1$ satisfying $\phi_{ij}(\lambda_i)=\epsilon$ and $\phi_{ij}(\lambda_j)=0.$} by Mult($\mathcal{H}_{\ell}$). Examples (to be found in subsection \ref{302}) include products of powers of $2$-point Pick kernels (Example \ref{821}) and Hardy spaces on finitely-connected planar domains (Example \ref{822}). In subsection  \ref{303}, we give a  general criterion for the AS property. The idea here (see Theorem \ref{833} for a precise statement) is that a kernel $\ell$ has the AS property if and only if any weakly separated finite union of ``sufficiently sparse" sequences forms an $\mathcal{H}_{\ell}$-interpolating sequence. As a consequence, we discover that an even larger number of well-studied spaces possess AS kernels. These include ``large" weighted Bergman spaces (Example \ref{1002}) and weighted Bargmann-Fock spaces (Example \ref{1003}). Subsection  \ref{303} culminates in Theorem \ref{1000}, which describes a large class of pairs $(s, \ell)$ for which Question \ref{4} has a positive answer (this includes all pairs $(s, \ell)$ such that $\ell$ is one of the kernels from the previous examples and $s$ is a complete Pick factor of $\ell$). \par
Finally, it should be noted that the pair $(s, \ell)$ constructed in the proof of Theorem \ref{8}, while offering  a counterexample to Question \ref{4}, is not a natural setting for the solution of interpolation problems. One might then wonder whether imposing a few weak regularity conditions (like the ones in the statement of Theorem \ref{833}) on $(s, \ell)$  would always force the pair to behave according to the manner predicted by Question \ref{4}. This doesn't seem to be the case. In particular, subsection \ref{304} contains the construction of a ``nice" holomorphic pair $(s, \ell)$ on the bidisk which provides us with a more natural counterexample to Question \ref{4}  (however, that construction is not sufficient to establish Theorem \ref{8} in its entirety).

\small
 \section{PRELIMINARIES} \label{2}

\large
\small
\subsection{Complete Pick kernels}
\large  
Let $\mathcal{H}_k$ be a reproducing kernel Hilbert space on a set $X$, with kernel $k$. For basic facts regarding the theory of reproducing kernels, see \cite{Pick} and \cite{PauRa}. Also, let $n$ be a positive integer, and let $\mathcal{M}_n$ denote the $n$-by-$n$ complex matrices. We say that $k$ has the \textit{$N$-point $\mathcal{M}_n$ Pick property} if, for
every finite sequence $\lambda_1, \dots, \lambda_N$ of $N$ distinct points in $X$, and every sequence $W_1, \dots, W_N$ in $\mathcal{M}_n$, positivity of the block matrix 
$$\bigg[k(\lambda_i, \lambda_j)(I_{\mathbb{C}^n}-W_iW^*_j)   \bigg]^N_{i,j=1} $$
implies the existence of a multiplier $\Phi$ of $\mathcal{H}_k\otimes\mathbb{C}^n$ of norm at most $1$ that satisfies
$$\Phi(\lambda_i)=W_i, \hspace{0.3 cm} \text{ for all }1\le i\le N.$$
When $n=1$, we say $k$ has the \textit{$N$-point  (scalar) Pick property}. If $k$ has the $N$-point $\mathcal{M}_n$ Pick property
for every $n$ and $N$, we say the kernel, and the corresponding Hilbert space $\mathcal{H}_k$, have
the \textit{complete Pick property}. \par 
Examples of such kernels and spaces (all proofs can be found in \cite{Pick}) are the Szeg{\H o} kernel $\frac{1}{1-z\overline{w}}$ for the Hardy space on the unit disk; the Dirichlet kernel $ \frac{-1}{z\overline{w}}\log(1-z\overline{w})$ on the
disk; the kernels $\frac{1}{(1-z\overline{w})^t}$ for $0<t<1$ on the disk; the Sobolev space $W^2_1$ on the unit interval;  and the Drury-Arveson space, the
space of analytic functions on the unit ball $\mathbb{B}_d$ of a $d$-dimensional Hilbert
space (where $d$ may be infinite) with kernel
$$k(z, w)=\frac{1}{1-\langle z, w\rangle}.$$
\par A kernel $k$ is said to be \textit{irreducible} if the underlying set $X$ cannot be partitioned into two non-empty disjoint sets $X_1, X_2$ so that $k(x_1, x_2)=0$
for all $x_1\in X_1, x_2\in X_2$. The kernel $k$ of an irreducible complete Pick space satisfies
$k(z, w)\neq 0$ for all $z, w\in X$; see \cite[Lemma 1.1]{CNPkernels}. By Theorem 3.1 of \cite{CNPkernels}, the
space $\mathcal{H}_k$ is an irreducible complete Pick space if and only if there exist a function
$f: X\to\mathbb{C}\setminus{0}$, a number $d\in\mathbb{N}\cup\{\infty\}$ and a function $b: X\to\mathbb{B}_d$, where $\mathbb{B}_d$ denotes the open unit ball of a $d$-dimensional Hilbert
space $\mathcal{K}$, so that
\begin{equation}  k_w(z)=\frac{f(z)\overline{f(w)}}{1-\langle b(z), b(w) \rangle_{\mathcal{K}}} \hspace{0.3 cm} (z, w\in X). \label{2.1} \end{equation}
(See also \cite{Knese} for a simple proof of necessity.) Finally, a kernel $k$ is normalized at $w_0\in X$ if $k(z,w_0)=1$ for all $z\in X$. One can always \textit{rescale} an irreducible complete Pick kernel (see \cite[Section 2.6]{Pick} for more
background on rescaling kernels) to achieve that in \ref{2.1} the function $f$ is the constant function $1$ and $b(w_0)=0.$
We again point out that working in normalized spaces is merely convenient, not essential for our proofs.

\small
\subsection{Complete Pick factors} \label{1030}
\large 
Suppose $k, \ell$ are reproducing kernels on $X$. Then,  $\text{Mult}(\mathcal{H}_k, \mathcal{H}_{\ell})$ is the collection of functions $\phi: X\to\mathbb{C}$ such that $(M_{\phi}f)(z)=\phi(z)f(z) $ defines a
bounded operator $M_{\phi}:\mathcal{H}_k\to \mathcal{H}_{\ell}$ . It is easy to see that for
$\phi\in\text{Mult}(\mathcal{H}_k, \mathcal{H}_{\ell})$ one has
$$M_{\phi}^*\ell_w=\overline{\phi(w)}k_w,$$
for all $w\in X$. Moreover, the multipliers $\phi$ with $||\phi||\le M$ are characterized (see \cite[Theorem 5.21]{PauRa}) by the positivity of
\begin{equation} \label{706}
M^2\ell_w(z)-\phi(z)\overline{\phi(w)}k_w(z). 
\end{equation}
We say that the pair $(k, \ell)$ has the \textit{Pick property} if, for every finite sequence of distinct points $\lambda_1, \dots, \lambda_N\in X$ and every sequence $w_1, \dots, w_N\in\mathbb{C}$, positivity of the matrix
\begin{equation}\big[\ell_{\lambda_i}(\lambda_j)-w_j\overline{w}_ik_{\lambda_i}(\lambda_j)\big]^N_{i, j=1} \label{2.2} \end{equation}
implies the existence of a multiplier $\phi\in\text{Mult}(\mathcal{H}_k, \mathcal{H}_{\ell})$ of norm at most 1 that
satisfies $$\phi(\lambda_i)=w_i, \hspace{0.3 cm} \text{ for all } 1\le i\le N.$$
Note that, as observed in \cite[Section 4]{InterpolatingPick}, if the pair $(k, \ell)$ has
the Pick property, then one can solve Pick problems with infinitely many points.  \par
Now, assume that $s$ is an irreducible complete Pick
kernel normalized at some point, hence
$$s_w(z)=\frac{1}{1-\langle b(z), b(w)\rangle},$$
where $b: X\to\mathbb{B}_d$. Assume also that $\ell$ is another kernel on $X$ such that $\ell/s$ is positive semi-definite (denoted by $\ell/s>>0$). Simple examples of such kernels are given by $\ell=s^t$, $t\ge 1$. Note that the positivity condition for $\ell/s$ is satisfied if and only if $b\in\text{Mult}(\ell\otimes \mathbb{C}^d, \ell)$ with $||M_{b}||\le 1$ (if $d=\infty$, $\mathbb{C}^d$ is treated as $\ell^{2}$), see \cite[Lemma 2.2]{Factorizationsinduced} for a proof. In recent years, kernels with a complete Pick factor have  been investigated in regard to invariant subspaces \cite{BeurlingLaxHalmos}, factorization theorems \cite{Factorizationsinduced}, \cite{freeouter} and the column-row property \cite[Section 3.8]{Column-Row}. \par 
The following result is very useful in the context of $\text{Mult}(\mathcal{H}_s, \mathcal{H}_{\ell})$-interpolating sequences. It appears as Proposition 4.4 in \cite{InterpolatingPick}, where it is proved as an application of Leech's theorem \cite[Theorem 8.57]{Pick}.
\begin{theorem} \label{21}
Suppose $\ell, s$ are kernels on $X$ such that $s$ has the complete Pick property and $\ell/s>>0$. Then, the pair $(s, \ell)$ has the Pick property.
\end{theorem}
An important consequence of Theorem \ref{21} is:
\begin{theorem}[Aleman, Hartz, M\raise.5ex\hbox{c}Carthy and Richter \cite{InterpolatingPick}] \label{3} 
Suppose $\ell, s$ are kernels on $X$ such that $s$ is a normalized complete Pick kernel and $\ell/s>>0$.
\begin{itemize}
    \item[(a)] A sequence is interpolating for $\text{Mult}(\mathcal{H}_{s}, \mathcal{H}_{\ell})$ if and only if it satisfies the Carleson measure condition (CM) for $\mathcal{H}_s$ and is interpolating for
$\mathcal{H}_{\ell}$.
\item[(b)] If a sequence is weakly separated by $s$, then it is interpolating for $\text{Mult}(\mathcal{H}_{s}, \mathcal{H}_{\ell})$
if and only if it is interpolating for $\text{Mult}(\mathcal{H}_{s})$.

\end{itemize}

\end{theorem}

Note that, in general, a $\text{Mult}(\mathcal{H}_{s}, \mathcal{H}_{\ell})$-interpolating sequence needn't be weakly separated by $s$ (\cite[Example 4.13]{InterpolatingPick}).

\small
\subsection{Grammians}
\large 

Let $\mathcal{H}_k$ be a reproducing kernel Hilbert space on a set $X$, with kernel $k$. We write $\hat{k}_z=k_z/||k_z|| $ for the normalized kernel function at $z$. Let $\{\lambda_i\}$ be a sequence of distinct points in $X$. The Grammian, or Gram matrix,
associated with the sequence is the (infinite) matrix $G(k)=[G_{i, j}]$, where
$$G_{i, j}=\langle \hat{k}_{\lambda_i}, \hat{k}_{\lambda_j} \rangle=\frac{k(\lambda_j, \lambda_i)}{\sqrt{k(\lambda_j, \lambda_j)k(\lambda_i, \lambda_i)}}.$$
We say that the sequence $\{\lambda_i\}\subset X$ has a bounded Grammian (BG) if the Gram matrix, thought of as an operator on $\ell^2$, is bounded; we shall
say that it is bounded below (BB) if the Gram matrix is bounded below on $\ell^2$. It is known that if the Grammian of a sequence for the Szeg{\H o} kernel $s_w(z)=\frac{1}{1-z\overline{w}} $ is bounded below, then it is bounded above. This is no longer true in the Dirichlet space; see \cite{Bishop}. Sequences satisfying (BB) have also been called \textit{simply interpolating} and have been studied in \cite{ArcozziRochbergDirichlet}, \cite{ChalmoukisSimply} and  \cite{ChalmoukisOnto} in the setting of the Dirichlet space. 
\par 
The following lemma is well-known (see \cite[Chapter 9]{Pick} for a proof). 
\begin{lemma}  \label{131} 
\begin{itemize}
\item[]  
 
    \item[(a)] The Grammian is bounded (BG) if and only if the sequence satisfies the Carleson measure condition (CM) for $\mathcal{H}_k$.
    \item[(b)] The following three conditions are equivalent:
    \begin{itemize}
        \item[(i)] the Grammian is bounded and bounded below (BG)+(BB),
        \item[(ii)] the functions $\hat{k}_{\lambda_i}$ form a Riesz sequence, i.e. there exist $c_1, c_2 >0$ such that for all scalars $a_i$, 
        $$c_1\sum_{i}|a_i|^2\le \big|\big|\sum_{i}a_i\hat{k}_{\lambda_i}\big|\big|^2\le c_2\sum_{i}|a_i|^2, $$
        \item[(iii)] the sequence is interpolating for $\mathcal{H}_k$ (IH). 
    \end{itemize}

\end{itemize}
\end{lemma}
We will also be making crucial use of the following result, which is part of \cite[Theorem 9.46]{Pick}. We use $\{e_i\}$ to denote the standard orthonormal basis for $\ell^2$.
\begin{theorem} \label{20}
Let $k$ be an irreducible complete Pick kernel on $X$, let $\{\lambda_i\}\subset X$, and let $G$ denote the Grammian associated with $\{\lambda_i\}$. Then, $G$ is bounded if and only if there exists is a multiplier
$\Psi\in\text{Mult}(H_s, H_s\otimes \ell^2)$ such that 
$$\Psi(\lambda_i)=e_i=\begin{pmatrix} \mbox{\hspace{0.1 cm} $\mathrm{*}$ \hspace{0.1 cm}} \\ \vdots \\  \vspace{0.06 cm} \mbox{$\mathrm{*}$} \\ 1  \vspace{0.07 cm} \\ \mbox{$\mathrm{*}$}\vspace{0.001 cm} \\ \vspace{0.001 cm}\vdots\end{pmatrix},
$$
for every $i.$
\end{theorem}
As noted in \cite{InterpolatingPick}, a more restrictive definition of irreducibility is used in the statement given in \cite{Pick}, however our more relaxed definition suffices for the proof to go through. \par
Finally, we record a basic Hilbert space lemma (as seen in \cite[Section I]{ShapiroShields}) which will be used repeatedly throughout the paper, often without special mention. 
\begin{lemma} \label{708}
Suppose $\mathcal{H}$ is a Hilbert space and $v_0, v_1, \dots, v_n\in\mathcal{H}$. Let $d$ denote the distance from $v_0$ to the subspace spanned by $v_1, \dots, v_n$. If $v_1, \dots, v_n$ are also linearly independent, then 
$$d^2=\frac{\det[\langle v_i, v_j \rangle]_{0\le i,j\le n}}{\det[\langle v_i, v_j \rangle]_{1\le i,j\le n}}.$$
\end{lemma}

 \section{A characterization of $\text{Mult}(\mathcal{H}_{s}, \mathcal{H}_{\ell})$-interpolating sequences} \label{5}

\large
\subsection{Necessary and sufficient conditions}
Suppose $s$ is a normalized (irreducible) complete Pick kernel defined on a set $X$. Suppose also that $\ell$ is another kernel on $X,$ satisfying $$\ell(z, w)=s(z, w)g(z, w),\hspace{0.3 cm} z,w \in X,$$
where $g$ is a kernel. Let $\{\lambda_i\}\subset X$ and $n\ge 2$. Recall that $\{\lambda_i\}$ is \textit{$n$-weakly separated by $\ell$} if there exists $\epsilon>0$ such that for every $n$-point subset $\{\mu_1, \dots, \mu_n\}\subset\{\lambda_i\}$ we have
$$\text{dist}\big(\hat{\ell}_{\mu_1}, \text{ span}\big\{\hat{\ell}_{\mu_2}, \dots, \hat{\ell}_{\mu_n}\big\} \big)\ge \epsilon.$$
Similarly, we say that $\{\lambda_i\}$ is \textit{strongly separated by $\ell$} if there exists $\epsilon>0$ such that for every $i\in\mathbb{N}$ we have
$$\text{dist}\big(\hat{\ell}_{\lambda_i}, \text{ span}_{j\neq i}\big\{\hat{\ell}_{\lambda_j}\big\} \big)\ge \epsilon.$$

The fact that the pair $(s, \ell)$ satisfies the Pick property allows us to recast weak and strong separation by $\ell$ in terms of separation by elements of $\text{Mult}(\mathcal{H}_{s}, \mathcal{H}_{\ell})$. 
\begin{lemma} \label{18}
Suppose $s$ is a normalized complete Pick factor of a kernel $\ell$ on $X.$  Also, let $\{\lambda_i\}\subset X$ and $n\ge 2.$
\begin{itemize}
    \item[(a)] $\{\lambda_i\}$ is $n$-weakly separated by $\ell$ if and only if there exists $\epsilon>0$ such that for every $n$-point subset $\{\mu_1, \mu_2, \dots, \mu_n\}$ of $\{\lambda_i\}$ there exists a multiplier $\phi\in\text{Mult}(\mathcal{H}_{s}, \mathcal{H}_{\ell})$ of norm at most $1$ with $\phi(\mu_1)=\epsilon\frac{||\ell_{\mu_1}||}{||s_{\mu_1}||}$ and $\phi(\mu_j)=0,$ for $j=2, 3,\dots, n$.
    \item[(b)] $\{\lambda_i\}$ is strongly separated by $\ell$ if and only if there exists $\epsilon>0$ such that for every $i\in\mathbb{N}$ there exists a multiplier $\phi\in\text{Mult}(\mathcal{H}_{s}, \mathcal{H}_{\ell})$ of norm at most $1$ with
    $\phi(\lambda_i)=\epsilon\frac{||\ell_{\lambda_i}||}{||s_{\lambda_i}||}$ and $\phi(\lambda_j)=0$ for every $j\neq i$.
    
\end{itemize}
\end{lemma}
\begin{proof}
First, we prove (a). Let $n\ge 2$ and suppose $\{\lambda_i\}$ is $n$-weakly separated by $\ell$. We can then find $\epsilon>0$ such that for every $n$-point subset $\{\mu_1, \dots, \mu_n\}\subset\{\lambda_i\}$ we have 
$$d=\text{dist}\big(\hat{\ell}_{\mu_1}, \text{ span}\big\{\hat{\ell}_{\mu_2}, \dots, \hat{\ell}_{\mu_n}\big\} \big)\ge \epsilon.$$
Now, fix $n$ points $\{\mu_1, \dots, \mu_n\}\subset\{\lambda_i\}$ and let $m\in\{2, 3,\dots, n\}$. $n$-weak separation implies that the vectors $\{\hat{\ell}_{\mu_1}, \hat{\ell}_{\mu_2}, \dots, \hat{\ell}_{\mu_n}\}$ are linearly independent. In view of Lemma \ref{708}, we can write
$$\frac{\det\big[\langle \hat{\ell}_{\mu_i}, \hat{\ell}_{\mu_j}  \rangle\big]_{1\le i, j\le m}}{\det\big[\langle \hat{\ell}_{\mu_i}, \hat{\ell}_{\mu_j}  \rangle\big]_{2\le i, j\le m}}=\Big[\text{dist}\big(\hat{\ell}_{\mu_1}, \text{ span}\big\{\hat{\ell}_{\mu_2}, \dots, \hat{\ell}_{\mu_m}\big\} \big)\Big]^2\ge d^2\ge\epsilon^2$$
$$\RR \det\big[(1-w_j\overline{w}_i)\langle \hat{\ell}_{\mu_i}, \hat{\ell}_{\mu_j}  \rangle\big]_{1\le i, j\le m}>0, $$
where $w_1=\epsilon/2$ and $w_2=w_3=\dots=w_n=0$. Since this is true for arbitrary $m\in\{2, 3,\dots, n\}$, Sylvester's criterion tells us that the matrix 
$$\big[(1-w_j\overline{w}_i)\langle \hat{\ell}_{\mu_i}, \hat{\ell}_{\mu_j}  \rangle\big]_{1\le i, j\le n}$$
is positive semi-definite. Multiplying the previous matrix by the dyad \\ $\big[||\ell_{\mu_i}||\cdot||\ell_{\mu_j}||\big]$, we obtain the positivity of
$$\big[(1-w_j\overline{w}_i) \ell(\mu_j, \mu_i)\big]_{1\le i, j\le n},$$
which can be rewritten as 
$$\big[\ell(\mu_j, \mu_i)-v_j\overline{v}_is(\mu_j, \mu_i)\big]_{1\le i, j\le n},$$
where $v_1=\frac{\epsilon}{2}\frac{||\ell_{\mu_1}||}{||s_{\mu_1}||}$ and $v_2=v_3=\dots=v_n=0$. But $(s, \ell)$ has the Pick property, so we can deduce the existence of a multiplier $\phi\in\text{Mult}(\mathcal{H}_{s}, \mathcal{H}_{\ell})$ of norm at most $1$ such that $\phi(\mu_1)=\frac{\epsilon}{2}\frac{||\ell_{\mu_1}||}{||s_{\mu_1}||}$ and $\phi(\mu_j)=0,$ for $j=2, 3,\dots, n$. \\
We have proved one implication from part (a). For the converse, simply reverse the steps in the previous proof (the Pick property of  $(s, \ell)$ is no longer necessary). \par The proof of (b) is essentially identical to that of (a). One point worth mentioning is that the inequalities $ \det\big[(1-w_j\overline{w}_i)\langle \hat{\ell}_{\mu_i}, \hat{\ell}_{\mu_j}  \rangle\big]_{1\le i, j\le m}>0$, for all $m\ge 2,$ allow us to deduce (through Sylvester's criterion and standard approximation arguments) the positivity of the infinite matrix $\big[(1-w_j\overline{w}_i)\langle \hat{\ell}_{\mu_i}, \hat{\ell}_{\mu_j}  \rangle\big]_{i, j}$. The rest of the proof carries over without change.
\end{proof}
Now, we use the column-row property for spaces with a complete Pick factor (see \cite[Theorem 3.18]{Column-Row}) to characterize Mult$(\mathcal{H}_s, \mathcal{H}_{\ell})$-interpolating sequences in terms of the $\mathcal{H}_s$-Carleson measure condition and strong separation by $\ell$. Our argument is motivated by the proof of Theorem 4.4 in \cite{Column-Row}. 
\begin{theorem}\label{30}
Suppose $s$ and $\ell$ are as above and let $\{\lambda_i\}\subset X$. Then, $\{\lambda_i\}$ is interpolating for Mult$(\mathcal{H}_s, \mathcal{H}_{\ell})$ if and only if it satisfies the Carleson measure condition for $\mathcal{H}_s$  and is strongly separated by $\ell.$  In this case, there exists a bounded linear operator of interpolation associated with $\{\lambda_i\}$.
\end{theorem}
\begin{proof} First, suppose that $\{\lambda_i\}$ is interpolating for Mult$(\mathcal{H}_s, \mathcal{H}_{\ell})$.
By Theorem \ref{3}(a), we obtain that $\{\lambda_i\}$ satisfies (CM) with respect to $\mathcal{H}_s$.  \\
Recall also (see (\ref{17})) that every $\phi\in\text{Mult}(\mathcal{H}_s, \mathcal{H}_{\ell})$ satisfies
$$|\phi(z)|\le ||\phi||_{\text{Mult}(\mathcal{H}_s, \mathcal{H}_{\ell})}\cdot\frac{||\ell_z||}{||s_z||},$$
for every $z\in X.$ We can thus define 
$$S: \text{Mult}(\mathcal{H}_s, \mathcal{H}_{\ell})\to\ell^{\infty}  $$
$$\phi\mapsto \bigg\{\phi(\lambda_i)\cdot\frac{||s_{\lambda_i}||}{||\ell_{\lambda_i}||} \bigg\}_{i\ge 1}.$$
$S$ is well-defined, linear and bounded by $||\phi||_{\text{Mult}(\mathcal{H}_s, \mathcal{H}_{\ell})}$. Since $\{\lambda_i\}$ is interpolating, $S$ is also onto $\ell^{\infty}.$ A standard application of the Open Mapping Theorem then allows us to deduce the existence of a constant $C>0$ \textit{(the constant of interpolation)}, such that for every $\{w_i\}\in\ell^{\infty}$ we can find $\phi\in \text{Mult}(\mathcal{H}_s, \mathcal{H}_{\ell})$ with the property that $\phi(\lambda_i)=w_i||\ell_{\lambda_i}||/||s_{\lambda_i}||$ and $||\phi||_{\text{Mult}(\mathcal{H}_s, \mathcal{H}_{\ell})}\le C\cdot||\{w_i\}||_{\infty}.$ In view of Lemma \ref{18}(b), this implies that $\{\lambda_i\}$ is strongly separated by $\ell.$
\par For the converse, suppose that $\{\lambda_i\}$ satisfies (CM) with respect to $s$  and is strongly separated by $\ell.$
\\ By Lemma \ref{18}(b), there exist multipliers $\{\phi_i\}\subset \text{Mult}(\mathcal{H}_s. \mathcal{H}_{\ell})$ and $M>0$ such that 
\begin{itemize}
    \item[(i)] $\phi_i(\lambda_j)=\delta_{ij}\frac{||\ell_{\lambda_i}||}{||s_{\lambda_i}||}$, for every $i, j$;
    \item[(ii)] $||\phi_i||_{\text{Mult}(\mathcal{H}_s, \mathcal{H}_{\ell})}\le M$, for every $i$.
\end{itemize}
Also, by Theorem \ref{20}, there exists a multiplier $\Psi\in\text{Mult}(\mathcal{H}_s, \mathcal{H}_s\otimes \ell^2)$ such that 
$$\Psi(\lambda_i)=\begin{pmatrix} \mbox{\hspace{0.1 cm} $\mathrm{*}$ \hspace{0.1 cm}} \\ \vdots \\  \vspace{0.06 cm} \mbox{$\mathrm{*}$} \\ 1  \vspace{0.07 cm} \\ \mbox{$\mathrm{*}$}\vspace{0.001 cm} \\ \vspace{0.001 cm}\vdots\end{pmatrix},
$$
for every $i.$\\
Consider now the bounded diagonal operator $\text{diag}\{\phi_1, \phi_2, \phi_3, \dots \}\in\text{Mult}(\mathcal{H}_s\otimes\ell^2, \mathcal{H}_{\ell}\otimes\ell^2)$ and define
$$\Phi:=\text{diag}\{\phi_1, \phi_2, \phi_3, \dots \}\cdot \Psi\in\text{Mult}(\mathcal{H}_s, \mathcal{H}_{\ell}\otimes \ell^2).$$
Notice that 
$$\Phi(\lambda_i)=\frac{||\ell_{\lambda_i}||}{||s_{\lambda_i}||}\begin{pmatrix} \mbox{\hspace{0.1 cm} 0 \hspace{0.1 cm}} \\ \vdots \\ \vspace{0.1 cm}  \mbox{0} \\ 1  \vspace{0.1 cm} \\ \mbox{0}\vspace{-0.1 cm} \\ \vspace{0.2 cm}\vdots\end{pmatrix}, $$
for every $i$. \cite[Theorem 3.18]{Column-Row} now tells us that $\Phi^{T}\in\text{Mult}(\mathcal{H}_s\otimes \ell^2, \mathcal{H}_{\ell}) $ and thus, letting $\Delta: \ell^{\infty}\to\text{Mult}(\mathcal{H}_s\otimes\ell^2)$ denote the embedding via diagonal operators, we can define 
$$T: \ell^{\infty}\to \text{Mult}(\mathcal{H}_s, \mathcal{H}_{\ell}) $$
$$\{w_i\}\mapsto \Phi^{T}\cdot \Delta(\{w_i\})\cdot\Psi.$$
$T$ is well-defined, bounded and also satisfies 
$$[T(\{w_i\})](\lambda_j)=w_j\frac{||\ell_{\lambda_j}||}{||s_{\lambda_j}||} ,$$
for every $j$. Thus, $T$ is the linear operator of interpolation with respect to $\{\lambda_i\}$ and our proof is complete.
\end{proof}

We are now ready to prove Theorem \ref{7}. Of critical
importance will be the observation (see \cite[Proposition 9.11]{Pick} for a proof) that every sequence $\{\lambda_i\}\subset X$ satisfying (CM) with respect to $\mathcal{H}_s$ can be written as a union of $n$ sequences that are weakly separated by $s,$ where $n$ is a finite integer. In this setting, it turns out that $n$-weak ((and not merely weak) separation by $\ell$ is \textit{precisely} what is missing for $\{\lambda_i\}$ to be $\text{Mult}(\mathcal{H}_s, \mathcal{H}_{\ell})$-interpolating.

\begin{proof}[Proof of Theorem \ref{7}]
By Theorem \ref{30}, every $\text{Mult}(\mathcal{H}_s, \mathcal{H}_{\ell})$-interpolating sequence satisfies (CM) with respect to $s$ and is strongly separated by $\ell,$ hence also $n$-weakly separated by $\ell$ for every $n\ge 2.$ \par
For the converse, suppose $\{\lambda_i\}\subset X$ satisfies (CM) with respect to $s$ and is $n$-weakly separated by $\ell$, for every $n\ge 2$. If $\{\lambda_i\}$ also happens to be weakly separated by $s$, then it must be interpolating for $\text{Mult}(\mathcal{H}_s)$ (and hence interpolating for $\text{Mult}(\mathcal{H}_s, \mathcal{H}_{\ell})$ by Theorem \ref{3}(b)). If not, then there exists $n\ge 2$ such that $\{\lambda_i\}$ can be written as a union of $n$ sequences that are weakly separated by $s$. In other words, there exist disjoint sequences $\{p^1_i\}, \{p^2_i\}, \dots, \{p^n_i\}\subset X $ and a number $0<c<1$ such that $\{\lambda_i\}=\cup_{k=1}\{p^{k}_j\}$
and also for every $i\neq j$, 
\begin{equation}\big|\big\langle \hat{s}_{p^1_i}, \hat{s}_{p^1_j} \big\rangle\big|^2, \text{ } \big|\big\langle \hat{s}_{p^2_i}, \hat{s}_{p^2_j} \big\rangle\big|^2,\text{ }\dots\text{ }, \text{ } \big|\big\langle \hat{s}_{p^n_i}, \hat{s}_{p^n_j} \big\rangle\big|^2\le c. \label{23} \end{equation}
Now, choose an arbitrary point from $\{\lambda_i\}$. Without loss of generality, we may choose a point $p^1_{m_1}$ from $\{p^1_i\}.$  Consider the pseudometric 
$$d_s(\lambda_1, \lambda_2)=\sqrt{1-|\langle\hat{s}_{\lambda_1}, \hat{s}_{\lambda_2} \rangle |^2}$$
associated with $\mathcal{H}_s$.
By (\ref{23}), we obtain that $d_s(p^k_i, p^k_j)\ge \sqrt{1-c},$ for every $i\neq j$ and every $k\in\{1, 2,\dots, n\}$. Now, for any fixed $k\in\{2, 3,\dots, n\}$ and $i\neq j,$ the fact that $d_s$ is a pseudometric implies
$$\sqrt{1-c}\le d_s(p^k_i, p^k_j)\le d_s(p^k_i, p^1_{m_1})+d_s(p^1_{m_1}, p^k_j). $$
Consequently, for every $k\in\{2, 3,\dots, n\}$, there exists at most one point $p^k_{m_k}$ such that $d_s(p^1_{m_1}, p^k_{m_k})<\frac{\sqrt{1-c}}{2}$ (if such a point does not exist, pick an arbitrary point of $\{p^k_i\}$ to be $p^k_{m_k}$). Hence, there exists $c'>0$ (depending only on $c$) such that for every $k\in\{1, 2,\dots, n\}$ and every $j\neq m_k$, we have 
\begin{equation} \big|\big\langle \hat{s}_{p^1_{m_1}}, \hat{s}_{p^k_j} \big\rangle\big|^2\le c'<1. \label{24} \end{equation}
Also, for every $k\in\{1, 2,\dots, n\}$ and every $j\neq m_k$, the Pick property of $s$ allows us to find a contractive multiplier $\phi^k_j\in \text{Mult}(\mathcal{H}_s)$ such that $\phi^k_j(p^k_{j})=0$ and \begin{equation} \phi^k_j(p^1_{m_1})=\sqrt{1-\big|\big\langle \hat{s}_{p^1_{m_1}}, \hat{s}_{p^k_j} \big\rangle\big|^2}. \label{941} \end{equation}
Consider now the product \begin{equation} \prod_{k=1} ^n\prod_{j\neq m_k}\phi^k_j. \label{71} \end{equation}
More precisely, we take any weak-star cluster point of the partial products. We thus obtain a contractive multiplier $\Phi\in\text{Mult}(\mathcal{H}_s)$ such that $\Phi(p^k_j)=0$, for every $j\neq m_k.$ \\
Now, since $\{\lambda_i\}$ is $n$-weakly separated by $\ell,$ Lemma \ref{18}(a) tells us that there exists a contractive multiplier $\Psi_{m_1, m_2,\dots, m_n}\in\text{Mult}(\mathcal{H}_s, \mathcal{H}_{\ell})$ such that  $\Psi_{m_1, m_2,\dots, m_n}(p^k_{m_k})=0$ for every $k\ge 2$ and also $\Psi_{m_1, m_2,\dots, m_n}(p^1_{m_1})=\epsilon\frac{\big|\big|\ell_{p^1_{m_1}}\big|\big|}{\big|\big|s_{p^1_{m_1}}\big|\big|},$ where the constant $\epsilon>0$ does not depend on the choice of points $p^1_{m_1}, \dots, p^n_{m_n}$. \\  Finally, we put 
$$\tilde{\Phi}=\Psi_{m_1, m_2,\dots, m_n}\cdot \Phi\in \text{Mult}(\mathcal{H}_s, \mathcal{H}_{\ell}).$$
This is a contractive multiplier that is zero at every point of the sequence $\{\lambda_i\}$ except $p^1_{m_1}$. We also have (by (\ref{941}))
$$\tilde{\Phi}(p^1_{m_1})=\Psi_{m_1, m_2,\dots, m_n}(p^1_{m_1})\cdot \Phi(p^1_{m_1})$$
\begin{equation} =\epsilon\frac{\big|\big|\ell_{p^1_{m_1}}\big|\big|}{\big|\big|s_{p^1_{m_1}}\big|\big|}\prod_{k=1} ^n\prod_{j\neq m_k}\sqrt{1-\big|\big\langle \hat{s}_{p^1_{m_1}}, \hat{s}_{p^k_j} \big\rangle\big|^2}.\label{26} \end{equation}
Now, the fact that $\{\lambda_i\}$ satisfies (CM) with respect to $s$ implies the existence of a constant $C>0$ such that for every $f\in \mathcal{H}_s$ the inequality 
$$\sum \frac{|f(\lambda_i)|^2}{||s_{\lambda_i}||^2}\le C ||f||^2_{\mathcal{H}_s} $$
holds true. Choosing $f=\hat{s}_{p^1_{m_1}},$ we obtain 
\begin{equation}
    \sum_{k=1}^n\sum_j \big|\big\langle \hat{s}_{p^1_{m_1}}, \hat{s}_{p^k_j} \big\rangle\big|^2 \le C.\label{27} \end{equation}
Combining (\ref{24}), (\ref{26}) and (\ref{27}), we can conclude that $$\frac{\big|\big|s_{p^1_{m_1}}\big|\big|}{\big|\big|\ell_{p^1_{m_1}}\big|\big|}\tilde{\Phi}(p^1_{m_1})$$ is bounded below by a positive number that only depends on the constants $\epsilon, c$ and $C$ and not on the specific point $p^1_{m_1}$ we started with. Thus, $\{\lambda_i\}$ is strongly separated by $\ell$. By Theorem \ref{30}, $\{\lambda_i\}$ must be $\text{Mult}(\mathcal{H}_s, \mathcal{H}_{\ell})$-interpolating. \end{proof}

\begin{remark}
In the setting of Theorem \ref{7}, suppose $\{\lambda_i\}\subset X$ is a union of $n$ disjoint sequences (where $n\ge2$) that are interpolating for $\text{Mult}(\mathcal{H}_s)$. The previous proof tells us that $\{\lambda_i\}$ is $\text{Mult}(\mathcal{H}_s, \mathcal{H}_{\ell})$-interpolating if and only if it is $n$-weakly separated by $\ell.$
\end{remark}

\subsection{A counterexample to Question \ref{4}}

Now, suppose we have a sequence $\{\lambda_i\}$ satisfying (CM) with respect to $s$ and suppose also that $n\ge 3$ is the smallest integer such that $\{\lambda_i\}$ can be written as a union of $n$ disjoint sequences that are weakly separated by $s$. Then, not even $(n-1)$-weak (let alone weak) separation by $\ell$ can, in general, guarantee that $\{\lambda_i\}$ is $\text{Mult}(\mathcal{H}_s, \mathcal{H}_{\ell})$-interpolating. This is essentially the content of Theorem \ref{8}, which we now prove.

\begin{proof}[Proof of Theorem \ref{8}]
Let $s$ be the Szeg{\H o} kernel on the unit disk $\mathbb{D}.$ For every $n\ge 3$, choose $n$ disjoint sequences $\{\lambda^{n, 1}_i\}_i, \{\lambda^{n, 2}_i\}_i, \dots, \{\lambda^{n, n}_i\}_i\subset\mathbb{D}$ that are interpolating for $\text{Mult}(\mathcal{H}_s)$ and such that their $i^{th}$ terms satisfy
\begin{equation}
    \lim_i d_s(\lambda^{n, k}_i, \lambda^{n, m}_i)=0, \label{39}
\end{equation}
for every $k, m\in\{1, 2, \dots, n\}$. We also require the existence of $c>0$ (which could depend on $n$) such that 
\begin{equation}  d_s(\lambda^{n, k}_i, \lambda^{n, m}_j)\ge c,\label{38} \end{equation}
 for all $n, i, j, k, m$ such that $i\neq j$. Finally, choose these sequences in such a way that $\lambda^{n, k}_i= \lambda^{\nu, m}_j$ is equivalent to $i=j, n=\nu$ and $k=m,$ i.e. there are no common points between them. Here is one way of constructing such sequences: let $\{a_1^1\}$ be an arbitrary point in $\DD$ and suppose that, for $m\ge 1,$ the $m$-point set $\{a^m_1, \dots, a^m_m\}$ has been determined. Then, choose $m+1$ points $a^{m+1}_1, \dots, a^{m+1}_{m+1}$ such that \begin{itemize}
     \item[(i)] $\frac{1-|a^{m+1}_i|}{1-|a^k_j|}\le \rho <1, $\hspace{0.2 cm} $\forall k\in\{1, \dots, m\}, i\in\{1, \dots, m+1\}, j\in\{1, \dots, k\}$.
     
     \item[(ii)] $d_s\big(a^{m+1}_i, a^{m+1}_j\big)\le \frac{1}{m+1},$\hspace{0.2 cm}  for all $i, j$.
 \end{itemize}
  Essentially, our sequences come in $m$-point packets (where $m$ is increasing) that are well-separated from one another but such that the points in each packet are increasingly close to each other. 
 Now, put $\{\lambda^{n, k}_i\}_{i\ge 1}=\{a^i_{c(n, k)}\}_{i\ge d(n)}$, where $c(n, k)=k+\sum^{n-1}_{j=1}j$ and $d(n)=\sum^{n}_{j=1}j$. Then, no two sequences will have any points in common. Also, by item (i), each $\{\lambda^{n, k}_i\}_{i\ge 1}$ converges to $\partial\DD$ exponentially and so must be interpolating for $\text{Mult}(\mathcal{H}_s)$. The same condition guarantees that (\ref{38}) must be valid for some constant $c>0$. Finally, item (ii) implies that (\ref{39}) is also satisfied.
 
 \par 
We now turn to the construction of the kernel $\ell$. For every $n\ge 3,$ choose $n$ linearly dependent vectors $\{v_{n, 1}, v_{n, 2},\dots, v_{n, n} \}$ in $\ell^2$ with the following property: any
choice of $n-1$ vectors among  $\{v_{n, 1}, v_{n, 2},\dots, v_{n, n} \}$ produces a linearly independent
set. Also, define $u: \mathbb{D}\to \ell^2$ by 
$$u(\lambda)=\begin{cases} 
v_{n, k}, & \text{ } \text{ if } \lambda=\lambda^{n, k}_i \text{ for some } i, n, k \\ \\
            e_1, & \text{ } \text{ for every other point } \lambda\in\mathbb{D}.
\end{cases}
$$
(The definition of $u$ on points other than $\lambda^{n, k}_i$ is not important.) Since $\lambda^{n, k}_i= \lambda^{\nu, m}_j$ is equivalent to $i=j, n=\nu$ and $k=m,$ the function $u$ is well-defined. Finally, consider the positive semi-definite kernel $g: \mathbb{D}\times\mathbb{D}\to\mathbb{C}$ given by $g(\lambda, \mu)=\langle u(\lambda), u(\mu)\rangle_{\ell^2}$ and put $\ell:=s\cdot g$.
\par Now, fix $n\ge 3$ (this will remain fixed for the rest of the proof). Since any $(n-1)$-point subset of $\{v_{n, 1}, v_{n, 2},\dots, v_{n, n} \}$ is linearly independent, there exists $\epsilon>0$ such that each vector $v_{n, j}$ always has distance greater than $\epsilon$ from the span of any $(n-2)$-point subset of the remaining $n-1$ vectors. Thus, we can deduce (in a manner identical to the proof of Lemma \ref{18}) the existence of $\epsilon_1>0$ such that for any $(n-1)$-point subset $\{\nu_1, \nu_2, \dots, \nu_{n-1}\}$ of $\{1, 2, \dots, n\}$ the matrix 
$$\big[ (1-w_j\overline{w}_i)\langle v_{n, \nu_j}, v_{n,\nu_i} \rangle \big]_{1\le i,j\le n-1},$$
is positive semi-definite, where $w_1=\epsilon_1$ and $w_2=\dots=w_{n-1}=0$. By definition of the kernel $g$, we obtain that for every such subset $\{\nu_1, \nu_2, \dots, \nu_{n-1}\}$ and every choice of $n-1$ (not necessarily distinct) integers $m_1, m_2, \dots, m_{n-1}$, the matrix 
$$\big[ (1-w_j\overline{w}_i) g(\lambda^{n,\nu_j}_{m_j}, \lambda^{n,\nu_i}_{m_i}) \big]_{1\le i,j\le n-1},$$
is positive semi-definite. Now, multiply with the corresponding Grammian associated to $s$ and apply the Schur product theorem to deduce positivity of the matrix 
    $$\big[ (1-w_j\overline{w}_i) \ell(\lambda^{n,\nu_j}_{m_j}, \lambda^{n,\nu_i}_{m_i}) \big]_{1\le i,j\le n-1}.$$
Since $(s, \ell)$ has the Pick property, this last positivity condition implies that for any $(n-1)$-point subset $\{\nu_1, \nu_2, \dots, \nu_{n-1}\}$ of $\{1, 2, \dots, n\}$ and every $m_1, m_2, \dots, m_{n-1}\ge 1$, there exists a contractive multiplier $\Phi\in\text{Mult}(\mathcal{H}_s, \mathcal{H}_{\ell})$ such that 
\begin{equation}
    \Phi(\lambda^{n, \nu_1}_{m_1})=\epsilon_1\frac{\big|\big|\ell_{\lambda^{n, \nu_1}_{m_1}}\big|\big|}{\big|\big|s_{\lambda^{n, \nu_1}_{m_1}}\big|\big|} \hspace{0.2 cm} \text{ and } \hspace{0.2 cm}\Phi(\lambda^{n, \nu_i}_{m_i})=0,
    \label{50}
\end{equation}
for all $i\in\{2, \dots, n-1\}.$  
\par Now, recall that each sequence $\{\lambda^{n,k}_i \}_i $ is interpolating for $\text{Mult}(\mathcal{H}_s)$ and hence their union $\cup_{k=1}^n\{\lambda^{n,k}_i\}$ must satisfy the Carleson measure condition for $\mathcal{H}_s$ (this is because of the elementary fact that the sum of two Carleson measures is a Carleson measure). In view of the separation condition (\ref{38}) and the Pick property of $s$, we can mimic the construction of the Blaschke-type multiplier (\ref{71}) from the proof of Theorem \ref{7} to deduce the existence of an $\epsilon_2>0$ with the property that, for every point $\lambda^{n, k}_i$, there exists a contractive multiplier $\Psi\in\text{Mult}(\mathcal{H}_s)$ such that 
\begin{equation}
    \Psi(\lambda^{n, k}_i)=\epsilon_2\hspace{0.2 cm} \text{ and } \hspace{0.2 cm}\Psi(\lambda^{n, m}_j)=0,
    \label{51}
\end{equation}
for all $m\in\{1, 2, \dots, n\}$ and every integer $j\neq i.$   \par 
Combining (\ref{50}) and (\ref{51}) (and also using the basic fact that the product of a function in $\text{Mult}(\mathcal{H}_s)$ with a function in $\text{Mult}(\mathcal{H}_s, \mathcal{H}_{\ell})$ yields a multiplier in $\text{Mult}(\mathcal{H}_s, \mathcal{H}_{\ell})$), we deduce that
for every $(n-1)$-point subset $\{\mu_1, \mu_2, \dots, \mu_{n-1}\}$ of the union $\cup_k\{\lambda^{n, k}_i\}$ there exists a multiplier $\phi\in\text{Mult}(\mathcal{H}_{s}, \mathcal{H}_{\ell})$ of norm at most $1$ such that $\phi(\mu_1)=\epsilon_1\epsilon_2\frac{||\ell_{\mu_1}||}{||s_{\mu_1}||}$ and $\phi(\mu_j)=0,$ for $j=2, 3,\dots, {n-1}$. By Lemma \ref{18}(a), we conclude that $\cup_k\{\lambda^{n, k}_i\}$ must be $(n-1)$-weakly separated by $\ell.$ 
\par We now show that $\cup_k\{\lambda^{n, k}_i\}$ is not $n$-weakly separated by $\ell.$ We proceed by contradiction; suppose instead that there exists $\epsilon>0$ such that for every 
$n$-point subset $\{\mu_1, \mu_2, \dots, \mu_n\}$ of $\cup_k\{\lambda^{n, k}_i\}$ the matrix 
$$\big[(1-w_j\overline{w}_i)\ell(\mu_i, \mu_j) \big]_{1\le i,j\le n} $$
is positive semi-definite, where $w_1=\epsilon$ and $w_2=\dots=w_n=0$. Choosing $\mu_k=\lambda^{n, k}_m$ gives us the positivity of
$$ \big[(1-w_j\overline{w}_i)\ell(\lambda^{n,i}_m, \lambda^{n,j}_m) \big]_{1\le i,j\le n}, $$
for every $m\ge 1.$ Next, multiply the previous matrix by the transpose of $[s(\lambda^{n,i}_m, \lambda^{n,j}_m)]$ (which must be positive semi-definite as well) and the dyad $[s(\lambda^{n,i}_m, \lambda^{n,i}_m)^{-1}s(\lambda^{n,j}_m, \lambda^{n,j}_m)^{-1}]$. The Schur product theorem then allows us to obtain
\begin{equation} \det\bigg[(1-w_j\overline{w}_i)g(\lambda^{n,i}_m, \lambda^{n,j}_m)\frac{|s(\lambda^{n,i}_m, \lambda^{n,j}_m)|^2}{s(\lambda^{n,i}_m, \lambda^{n,i}_m)s(\lambda^{n,j}_m, \lambda^{n,j}_m)}        \bigg]_{1\le i,j\le n}\ge 0. \label{72} \end{equation}

However, condition (\ref{39}) implies that $\lim_m  \frac{|s(\lambda^{n,i}_m, \lambda^{n,j}_m)|^2}{s(\lambda^{n,i}_m, \lambda^{n,i}_m)s(\lambda^{n,j}_m, \lambda^{n,j}_m)} =1,$ for every $i, j\in\{1, \dots, n\}$. Thus, letting $m\to\infty$ in (\ref{72}) and using the definition of $g$, we obtain
$$\det\big[(1-w_j\overline{w}_i)\langle v_{n, i}, v_{n, j}\rangle \big]_{1\le i,j\le n}\ge 0,$$
which implies 
$$\det\big[\langle v_{n, i}, v_{n, j}\rangle \big]_{1\le i,j\le n}\ge \epsilon^2||v_{n,_1}||^2\det\big[\langle v_{n, i}, v_{n, j}\rangle \big]_{2\le i, j\le n}>0,$$
a contradiction, as the vectors $\{v_{n, 1}, \dots,  v_{n, n}\}$ are linearly dependent. \par 
To sum up, we have showed that, for every $n\ge 3$, the sequence $\cup_{k=1}^n\{\lambda^{n, k}_i\}$ satisfies (CM) with respect to $s$ and is $(n-1)$-weakly separated, but not $n$-weakly separated by $\ell$. The proof is complete.
\end{proof}

\begin{remark}
Notice that the choice of the Szeg{\H o} kernel $s$ in our counterexample is not really important; all that was required for the proof to go through was a complete Pick kernel $s$ containing, for every $n\ge 3$, disjoint sequences $\{\lambda^{n, 1}_i\}_i, \{\lambda^{n, 2}_i\}_i, \dots, \{\lambda^{n, n}_i\}_i $ in the underlying set that satisfy (\ref{39}) and (\ref{38}).
\end{remark}
\begin{remark}
Here is a simpler counterexample (which only works for  \textit{fixed} $n$). Choose $n\ge 3$ and consider the kernel $s(\lambda, \mu)=\frac{1}{1-z^n\overline{w}^n}$ defined on $\DD\times\DD$. Also, let $\omega_1, \dots, \omega_n$ denote the $n^{th}$ roots of unity and choose $n$ vectors $v_1, \dots, v_n$ in $\mathbb{C}^{n-1}$ with the property that every choice of $n-1$ vectors among them produces a linearly independent set. Put
$$u(\lambda)=\begin{cases} 
v_{k}, & \text{ } \text{ if } \lambda=\frac{1}{2}\omega_k  \\ \\
            v_1, & \text{ } \text{ for every other point } \lambda\in\mathbb{D},
\end{cases}
$$
and set $g(\lambda, \mu)=\langle u(\lambda), u(\mu))\rangle$, $\ell:=s\cdot g$ and $z_k=\frac{1}{2}\omega_k$. By assumption, there exists $\epsilon>0$ such that the matrix 
$$\big[ (1-w_j\overline{w}_i)g(z_{\mu_i}, z_{\mu_j})\big]_{1\le i,j \le n-1} $$
is positive semi-definite for every $(n-1)$-point subset $\{\mu_1, \dots, \mu_{n-1}\}$ of $\{1, \dots, n\}$, where $w_1=\epsilon$ and $w_2=\dots=w_{n-1}=0$. But since $s(z_i, z_j)=4/3$ for every $i, j$, we immediately obtain the positivity of the matrix 
$$\big[ (1-w_j\overline{w}_i)\ell(z_{\mu_i}, z_{\mu_j})\big]_{1\le i,j \le n-1} .$$

Thus, the sequence $\{z_1, \dots, z_n\}$ is $(n-1)$-weakly separated by $\ell$ and also (trivially) satisfies the Carleson measure condition for $\mathcal{H}_s$. However, the kernel functions $\ell_{z_1}, \dots, \ell_{z_n}$ are linearly dependent, which implies that $\{z_1, \dots, z_n\}$ is not $n$-weakly separated. 

\end{remark}
 \begin{remark}
 Given the artificial nature of the kernel $g$, one might wonder whether more natural counterexamples to Question \ref{4} can be found. Subsection \ref{304} contains a ``nicer" counterexample involving holomorphic kernels on the bidisk. 
 \end{remark}

 \section{When is (CM)+(WS) sufficient?} \label{6}

 Once more, suppose that $s, \ell$ are two kernels on $X$ such that $\ell/s>>0$ and $s$ is a normalized complete Pick kernel. The proof of Theorem \ref{8} shows that a potential reason for the failure of the $\mathcal{H}_s$-Carleson measure condition and weak separation by $\ell$ to always be sufficient conditions for $\text{Mult}(\mathcal{H}_s. \mathcal{H}_{\ell})$-interpolation is that $X$ might contain weakly separated (by $\ell$) sequences that are not $n$-weakly separated by $\ell$ for some $n\ge 3.$ This motivates the following definition. 
 \begin{definition}
 Let $\ell$ be a kernel on a set $X.$ $\ell$ will be said to have the \textit{automatic separation property} if any sequence $\{\lambda_i\}\subset X$ that is weakly separated by $\ell$ must always be $n$-weakly separated by $\ell$ for every $n\ge 3.$ Such kernels will also be called $\textit{AS kernels.}$
 \end{definition}
 \begin{remark} The AS property is not very intuitive from a geometric viewpoint. Indeed, if $\ell$ has the automatic separation property, then, for every $n\ge 2$, a normalized kernel function $\hat{\ell}_{\mu}$ can be ``close" to the span of $n$ other normalized kernel functions if and only if it is actually ``close" to one of them (see \cite[Section 6]{Albertothesis} for a direct proof of the fact that the Szeg{\H o} kernel on $\DD$ has this property). Nevertheless, as we shall see in Section \ref{301}, it turns out that a surprisingly large number of well-known function spaces possess AS kernels.

\end{remark}
 As an immediate consequence of Theorem \ref{7}, we obtain that Question \ref{4} has a positive answer for every pair $(s, \ell)$ such that $\ell$ has the automatic separation property.
 \begin{corollary}\label{507}
 Let $s, \ell$ be two kernels on a set $X$ such that $\ell$ is an AS kernel and $s$ is a normalized complete Pick factor of $\ell.$ 
 Then, a sequence $\{\lambda_i\}\subset X$
is interpolating for $\text{Mult}(\mathcal{H}_{s}, \mathcal{H}_{\ell})$ if and only if it satisfies the Carleson measure condition for $\mathcal{H}_s$ and is weakly separated by $\ell$.
 \end{corollary}

 We now show that, under some additional weak hypotheses on the kernels $s$ and $\ell$, Question \ref{4} having a positive answer for the pair $(s, \ell)$ is actually \textit{equivalent} to $\ell$ being an AS kernel.
 
 \begin{theorem} \label{89}
 Suppose $X$ is a topological space, $\ell$ a kernel on $X$ with a normalized complete Pick factor $s$ and the following properties are satisfied:
 \begin{itemize}
 
      \item[(Q1)] $\ell: X\times X\to\mathbb{C}$ is continuous;
      \item[(Q2)] If $\{\lambda_i\}\subset X$ satisfies $||\ell_{\lambda_i}||\to\infty$, then $||\ell_{\lambda_i}||^{-1}\ell(\lambda_i, \mu)\to 0$ for every $\mu\in X$;
      \item[(Q3)] Let $\{\lambda_i\}\subset X$. Then,  either $||\ell_{\lambda_i} ||\to\infty$ or $\{\lambda_i\}$ contains a subsequence converging to a point inside $X$;
      \item[(Q4)] If $||\ell_{\lambda_i} ||\to\infty$, then $||s_{\lambda_i} ||\to\infty$.
       \end{itemize} 
      Then, the following assertions are equivalent:
      \begin{itemize}
          \item[(i)] For every $\{\lambda_i\}\subset X$,  $\{\lambda_i\}$ is interpolating for $\text{Mult}(\mathcal{H}_s. \mathcal{H}_{\ell})$ if and only if it satisfies the Carleson measure condition for $\mathcal{H}_s$ and is weakly separated by $\ell$. 
          \item[(ii)]  
          $\ell$ has the automatic separation property.
       \end{itemize}

 \end{theorem}
\begin{proof}
Let $\ell$ be a kernel with a complete Pick factor $s$ defined on a topological space $X$ such that properties (Q1)-(Q4) are all satisfied.  \par 
If (ii) holds, then (i) must also hold by Corollary \ref{507}. \par 
Now, suppose that (ii) fails. Thus, there exists a sequence $\{\lambda_i\}$ and $n\ge 2$ such that $\{\lambda_i\}$ is $n$-weakly separated but not $(n+1)$-weakly separated by $\ell$. This implies the existence of $n+1$ subsequences $\{\lambda^1_m\}_m, \{\lambda^2_m\}_m, \dots, \{\lambda^{n+1}_m\}_m\subset \{\lambda_i\}$ (which may contain repeated points) such that the points $\lambda^1_m, \lambda^2_m, \dots, \lambda^{n+1}_m$ are distinct from one another, for every $m$, and also   
\begin{equation}\lim_m\text{dist}\big(\hat{\ell}_{\lambda^1_m}, \text{ span}\big\{\hat{\ell}_{\lambda^2_m}, \dots, \hat{\ell}_{\lambda^{n+1}_m}\big\} \big)=0. \label{81}
\end{equation}
Since 
$$\Big[\text{dist}\big(\hat{\ell}_{\lambda^1_m}, \text{ span}\big\{\hat{\ell}_{\lambda^2_m}, \dots, \hat{\ell}_{\lambda^{n+1}_m}\big\} \big)\Big]^2=\frac{\det\big[\langle \hat{\ell}_{\lambda^i_m}, \hat{\ell}_{\lambda^j_m} \rangle \big]_{1\le i, j\le n+1}}{\det\big[\langle \hat{\ell}_{\lambda^i_m}, \hat{\ell}_{\lambda^j_m} \rangle \big]_{2\le i, j\le n+1}}, $$
condition (\ref{81}) (and the fact that all determinants involved are uniformly bounded) implies that 
\begin{equation}\lim_m \det\big[\langle \hat{\ell}_{\lambda^i_m}, \hat{\ell}_{\lambda^j_m} \rangle \big]_{1\le i, j\le n+1}=0. \label{82}
\end{equation}

By $n$-weak separation, we also obtain the existence of $\epsilon>0$ such that for every $n$-point subset $\{\mu_1, \dots, \mu_n\}$ of $\{\lambda_i\}$ we have 
 \begin{equation}\text{dist}\big(\hat{\ell}_{\mu_1}, \text{ span}\big\{\hat{\ell}_{\mu_2}, \dots, \hat{\ell}_{\mu_n}\big\} \big)>\epsilon. \label{83}
\end{equation}
We now use (Q1) and (Q2) to deduce a useful lemma about the behavior of the pseudometric $d_{\ell}$. 
\begin{lemma} \label{85}
Let $\ell$ be a kernel on $X$ satisfying (Q1) and (Q2). If $\{w_i\}$ and $\{z_i\}$ are two sequences in $X$ such that $w_i\to w$ for some $w\in X$ and $||\ell_{z_i}||\to\infty$, then 
$$d_{\ell}(w_i, z_i)\to 1.$$
\end{lemma}
\begin{proof}[Proof of Lemma \ref{85}]
Since $w_i\to w,$ continuity of $\ell$ (property (Q1)) implies that $d_{\ell}(w_i, w)\to 0.$ Also, by (Q2), we obtain that $\langle \hat{\ell}_{z_i}, f\rangle\to 0$ whenever $f$ is a finite linear combination of kernel functions. But since linear combinations of kernel functions are dense in $\mathcal{H}_{\ell},$ it must be true that $\hat{\ell}_{z_i}\to 0$ weakly in $\mathcal{H}_{\ell}$. Finally, $d_{\ell}$ is a pseudometric and so we can write 
$$-d_{\ell}(w, w_i)+d_{\ell}(w, z_i)\le d_{\ell}(w_i, z_i)\le 1,$$
where $d_{\ell}(w, w_i)\to 0$ and $d_{\ell}(w, z_i)=\sqrt{1-|\langle \hat{\ell}_{w}, \hat{\ell}_{z_i}  \rangle|^2}\to 1,$ as $\hat{\ell}_{z_i}\to 0$ weakly. Thus,
$d_{\ell}(w_i, z_i)\to 1$ and the proof is complete.
\end{proof}
We can also assume, without loss of generality, that each subsequence $\{\lambda^k_m\}_m$ either satisfies $||s_{\lambda^k_m}||\to\infty$ and $||\ell_{\lambda^k_m}||\to\infty$ or converges to a point $p_k\in X$. This is possible because of (Q3) and (Q4) (we can keep extracting subsequences until we have the desired properties). There are now three separate cases to examine.
 \par
 First, suppose that all of the sequences $\{\lambda^k_m\}_m$ converge to points $p_k$ in $X$. Combining (\ref{82}) with the continuity of $\ell$ (property (Q1)), we deduce that 
 
 \begin{equation} \det\big[\langle \hat{\ell}_{p_i}, \hat{\ell}_{p_j} \rangle \big]_{1\le i, j\le n+1}=0. 
 \label{84}
  \end{equation}
Hence, the kernel functions $\ell_{p_1}, \dots, \ell_{p_{n+1}}$ are linearly dependent. Consider the finite sequence $\{p_1, \dots, p_{n+1}\}\subset X$. This sequence trivially satisfies the Carleson measure condition for $\mathcal{H}_s$, is $n$-weakly separated (because of (\ref{83})), and hence weakly separated, but not $(n+1)$-weakly separated by $\ell$ (because of (\ref{84})). Thus, (i) fails in this case. \par
 Suppose now that at least one sequence $\{\lambda^k_m\}_m$ satisfies $||s_{\lambda^k_m}||\to\infty$ and $||\ell_{\lambda^k_m}||\to\infty$.  Suppose also that at least one sequence $\{\lambda^{k'}_m\}_m$ converges to a point $p_{k'}\in X$. We will arrive at a contradiction. Without loss of generality, we can reindex our sequences so that $\{\lambda^1_m\}, \dots, \{\lambda^r_m\}$ converge to points $p_1, \dots, p_r$ in $X,$ while $\{\lambda^{r+1}_m\}, \dots, \{\lambda^{n+1}_m\}$ satisfy $||s_{\lambda^k_m}||\to\infty$ and $||\ell_{\lambda^k_m}||\to\infty$ (where $1\le r\le n$). In view of (\ref{83}), we can write 
 \begin{equation}
     \det\big[\langle \hat{\ell}_{\lambda^i_m}, \hat{\ell}_{\lambda^j_m} \rangle \big]_{1\le i, j\le r}\ge \epsilon^2 \det\big[\langle \hat{\ell}_{\lambda^i_m}, \hat{\ell}_{\lambda^j_m} \rangle \big]_{2\le i, j\le r}\ge \dots \ge \epsilon^{2(r-1)}. 
     \label{86}
 \end{equation}
 Similarly, we obtain 
  \begin{equation}
     \det\big[\langle \hat{\ell}_{\lambda^i_m}, \hat{\ell}_{\lambda^j_m} \rangle \big]_{r+1\le i, j\le n+1} >\epsilon^{2(n-r)}. 
     \label{87}
 \end{equation}
 Now, Lemma \ref{85} tells us that (as $m\to\infty$)
   \begin{equation}
    \langle \hat{\ell}_{\lambda^i_m}, \hat{\ell}_{\lambda^j_m}\rangle \to 0, \hspace{0.2 cm} \forall i\in\{1, \dots, r\}, \text{   } \forall j\in\{r+1, \dots, n+1\}.
     \label{88}
 \end{equation}
 We can thus write 
 $$\det\big[\langle \hat{\ell}_{\lambda^i_m}, \hat{\ell}_{\lambda^j_m} \rangle \big]_{1\le i, j\le n+1}=$$ $$=\Big(\det\big[\langle \hat{\ell}_{\lambda^i_m}, \hat{\ell}_{\lambda^j_m} \rangle \big]_{1\le i, j\le r}\Big)\Big(\det\big[\langle \hat{\ell}_{\lambda^i_m}, \hat{\ell}_{\lambda^j_m} \rangle \big]_{r+1\le i, j\le n+1}\Big)+e_m,$$
 where $e_m\to 0$ as $m\to\infty$ because of (\ref{88}). Hence, letting $m\to\infty$ in the previous equality and using (\ref{86}) and (\ref{87}) gives us 
 $$\liminf_m \det\big[\langle \hat{\ell}_{\lambda^i_m}, \hat{\ell}_{\lambda^j_m} \rangle \big]_{1\le i, j\le n+1}\ge \epsilon^{2(n-1)}>0,$$
 a contradiction.  \par 
 Finally, suppose that all of our sequences $\{\lambda^k_m\}_m$ satisfy $||\ell_{\lambda^k_m}||\to\infty$ and  $||s_{\lambda^k_m}||\to\infty$. By \cite[Proposition 5.1]{InterpolatingPick}, we can extract a subsequence $\{\lambda^k_{m_j}\}_j$ from each $\{\lambda^k_m\}_m$  that is interpolating for $\text{Mult}(\mathcal{H}_s)$. Hence, the union $\cup_{k=1}^{n+1}\{\lambda^k_{m_j}\}$ will satisfy the Carleson measure condition for $\mathcal{H}_s$. By assumption, it will also be $n$-weakly separated but not $(n+1)$-weakly separated by $\ell$. Thus, (i) fails and our proof is complete.
\end{proof}
\begin{remark}
Property (Q2) is definitely satisfied whenever each kernel function $\ell_{\mu}$ is bounded on $X$, however this is not necessary in general. For instance, the kernel $\ell(\lambda, \mu)=e^{\lambda\overline{\mu}}$ of the Bargmann-Fock space on $\mathbb{C}$ satisfies (Q2), even though not every $\ell_{\mu}$ is a bounded function.
\end{remark}
\begin{remark}
There is no mention of the kernel $s$ in the statement of (ii). Thus, the answer to Question \ref{4} is entirely independent of the complete Pick factor $s,$ at least for pairs $(s, \ell)$ satisfying (Q1)-(Q4).
\end{remark}

 Before we proceed, we record the following useful lemma. It essentially says that factoring a kernel does not increase the distance between the normalized kernel functions. 
\begin{lemma} \label{944}
Suppose $g, \ell$ are two reproducing kernels on $X$ such that $\ell/g$ is positive semi-definite. Then, 
$$\text{dist}\big(\hat{g}_{\mu_1}, \text{ span}\big\{\hat{g}_{\mu_2}, \dots, \hat{g}_{\mu_n}\big\} \big)\le \text{dist}\big(\hat{\ell}_{\mu_1}, \text{ span}\big\{\hat{\ell}_{\mu_2}, \dots, \hat{\ell}_{\mu_n}\big\} \big), $$
for any $n$-point subset  $\{\mu_1, \mu_2, \dots, \mu_n\}$ of $X$ (where $n\ge 2$).
\end{lemma}
\begin{proof} 
Assume that there exists $\epsilon>0$ such that $$\text{dist}\big(\hat{g}_{\mu_1}, \text{ span}\big\{\hat{g}_{\mu_2}, \dots, \hat{g}_{\mu_n}\big\} \big)> \epsilon.$$
In view of Lemma \ref{708}, we obtain 
$$\frac{\det[\langle \hat{g}_{\mu_i},  \hat{g}_{\mu_j} \rangle]_{1\le i,j\le m}}{\det[\langle \hat{g}_{\mu_i},  \hat{g}_{\mu_j} \rangle]_{2\le i,j\le m}}>\epsilon^2, \hspace{0.3 cm} \text{ for all }m\in\{2, \dots, n\}.$$
Hence, from Sylvester's Criterion, we can deduce the positivity of the matrix
$$[(1-w_j\overline{w}_i)\langle \hat{g}_{\mu_i},  \hat{g}_{\mu_j} \rangle]_{1\le i,j\le n},$$
where $w_1=\epsilon$ and $w_2=\dots=w_n=0.$ Taking the Schur product with the positive semi-definite matrix $[\langle \hat{\ell}_{\mu_i},  \hat{\ell}_{\mu_j}\rangle/\langle \hat{g}_{\mu_i}, \hat{g}_{\mu_j}\rangle]$ then gives us

$$[(1-w_j\overline{w}_i)\langle \hat{\ell}_{\mu_i},  \hat{\ell}_{\mu_j} \rangle]_{1\le i,j\le n}>>0,$$
which implies that 
$$\text{dist}\big(\hat{\ell}_{\mu_1}, \text{ span}\big\{\hat{\ell}_{\mu_2}, \dots, \hat{\ell}_{\mu_n}\big\} \big)=\sqrt{\frac{\det[\langle \hat{\ell}_{\mu_i},  \hat{\ell}_{\mu_j} \rangle]_{1\le i,j\le n}}{\det[\langle \hat{\ell}_{\mu_i},  \hat{\ell}_{\mu_j} \rangle]_{2\le i,j\le n}}}\ge \epsilon.$$
This concludes the proof of the lemma.
\end{proof}
Let us now look at an example of a pair $(s, \ell)$ for which Question \ref{4} has a positive answer but such that $\ell$ is not an AS kernel.
\begin{exmp}
Let $n\ge 2$. Consider the restricted Szeg{\H o} kernel $s(z, w)=\frac{1}{1-z\overline{w}}$ on $\frac{1}{2}\DD$ and choose $n+1$ disjoint sequences $\{\lambda^1_m\}, \{\lambda^2_m\}, \dots, \{\lambda^{n+1}_m\}\subset\frac{1}{2}\mathbb{D}$ that converge to the boundary of $\frac{1}{2}\DD$ and such that their $m^{th}$ terms satisfy
\begin{equation}
    \lim_m \langle \hat{s}_{\lambda^i_m}, \hat{s}_{\lambda^j_m} \rangle=1, \label{90}
\end{equation}
for every $i, j$. Also, choose an orthonormal sequence $\cup_{1\le k\le n+1}\{e^k_m\}_m$ in $\ell^2$ and define $u:\frac{1}{2}\DD\to\ell^2$ by

$$u(\lambda)=\begin{cases} 
e^k_m, & \text{ } \text{ if } \lambda=\lambda^k_m \text{ for some } m\ge 1 \text{ and } k\in\{1, \dots,  n\}\\ \\
            \sum_{i=1}^n e^i_m+\frac{1}{m}e^{n+1}_m, & \text{ } \text{ if } \lambda=\lambda^{n+1}_m \text{ for some } m\ge 1 \\ \\
             e^1_1, & \text{ } \text{ for every other point $\lambda$. }
\end{cases}
$$
(The definition of $u$ on points different from $\lambda^k_m$ is not important.) Finally, set $g: \frac{1}{2}\DD\times\frac{1}{2}\DD\to\mathbb{C}$ to be equal to $g(\lambda, \mu)=\langle u(\lambda), u(\mu)\rangle$ and define $\ell:=s\cdot g$. Note that $\ell$ does not satisfy (Q3) from Theorem \ref{89}. \par
Consider now a sequence $\{\nu_i\} \subset \frac{1}{2}\DD$
that satisfies the Carleson measure condition for $\mathcal{H}_s$. The only way this can happen is if $\{\nu_i\}$ is actually a finite sequence (this is because $||s_{\nu_i}|| $ is uniformly bounded above). But then, $\{\nu_i\}$  will (trivially) be interpolating for $\text{Mult}(\mathcal{H}_s)$ and hence also for $\text{Mult}(\mathcal{H}_s, \mathcal{H}_{\ell})$. Thus, Question \ref{4} has a positive answer for the pair $(s, \ell)$.  \par 
We now show that $\ell$ is not an AS kernel. By definition of $g,$ we obtain that 
$\text{dist}\big(\hat{g}_{\mu_1}, \text{ span}\{\hat{g}_{\mu_2}, \dots, \hat{g}_{\mu_n}\}\big) $
is uniformly bounded below, where $\{\mu_1, \dots, \mu_n \}$ is any $n$-point subset of $\cup_k \{\lambda^k_m\}$.
In view of Lemma \ref{944}, this implies that   $\cup_k \{\lambda^k_m\}$ is $n$-weakly separated by $\ell$. However, notice that 
 $$\frac{\det\big[\langle g_{\lambda^i_m}, g_{\lambda^j_m} \rangle \big]_{1\le i,j\le n+1}}{\det\big[\langle g_{\lambda^i_m}, g_{\lambda^j_m} \rangle \big]_{1\le i,j\le n}} $$
 $$=\big[\text{dist}\big(g_{\lambda^{n+1}_m}, \text{ span}\{g_{\lambda^1_m}, \dots, g_{\lambda^n_m}\}\big)\big]^2 $$
 $$=\Bigg[\text{dist}\Bigg(\bigg(\sum_{i=1}^n e^i_m+\frac{1}{m}e^{n+1}_m\bigg), \text{ span}\{e^1_m, \dots, e^n_m\}\Bigg)\Bigg]^2\to 0, $$
 as $m\to \infty.$
Thus, $\det\big[\langle \hat{g}_{\lambda^i_m}, \hat{g}_{\lambda^j_m} \rangle \big]_{1\le i,j\le n+1}\to 0$ and so, in view of (\ref{90}), we obtain  $\det\big[\langle \hat{\ell}_{\lambda^i_m}, \hat{\ell}_{\lambda^j_m} \rangle \big]_{1\le i,j\le n+1}\to 0$ as $m\to\infty.$ But then, 
$$\big[\text{dist}\big(\hat{\ell}_{\lambda^{n+1}_m}, \text{ span}\{\hat{\ell}_{\lambda^1_m}, \dots, \hat{\ell}_{\lambda^n_m}\}\big)\big]^2=\frac{\det\big[\langle \hat{\ell}_{\lambda^i_m}, \hat{\ell}_{\lambda^j_m} \rangle \big]_{1\le i,j\le n+1}}{\det\big[\langle \hat{\ell}_{\lambda^i_m}, \hat{\ell}_{\lambda^j_m} \rangle \big]_{1\le i,j\le n}}\to 0,$$
as $m\to\infty$ (note that the determinant in the denominator is uniformly bounded below by $n$-weak separation). Hence, $\cup_k \{\lambda^k_m\}$ is not $(n+1)$-weakly separated by $\ell$, which implies that $\ell$ does not have the automatic separation property.

\end{exmp}

 \section{Which kernels have the automatic separation property?} \label{301}

\subsection{Separation by multipliers} \label{302}
In this subsection, we look at kernels enjoying a separation property which is stronger than the AS property. To be precise, we will be concerned with kernels $\ell$ defined on a set $X$ such that for every $\{\lambda_i\}\subset X,$ weak separation by $\ell$ implies weak separation by $\text{Mult}(\mathcal{H}_{\ell})$. 

\begin{definition} \label{832}
Let $\ell$ be a kernel on a set $X$. We will say that $\ell$ has the \textit{multiplier separation property} if, for every $\delta>0,$ there exists an $\epsilon>0$ such that, for any two points $\lambda_i\neq \lambda_j$ in $X$ satisfying $d_s(\lambda_i, \lambda_j)>\delta$, there exists $\phi_{ij}\in\text{Mult}(\mathcal{H}_{\ell})$ of norm at most $1$ satisfying $\phi_{ij}(\lambda_i)=\epsilon$ and $\phi_{ij}(\lambda_j)=0$.

\end{definition}
\begin{remark}
For any kernel $\ell$ and any sequence $\{\lambda_i\}\subset X$, weak separation by $\text{Mult}(\mathcal{H}_{\ell})$ always implies weak separation by $\ell$. This is a consequence of the positivity of (\ref{706}) (for $k=\ell$). Hence, for kernels satisfying the multiplier separation property, weak separation by the kernel always coincides with weak separation by the multiplier algebra.
\end{remark}

\par First, we show that the multiplier separation property does indeed imply the AS property.

\begin{proposition} \label{101}
Let $\ell$ be a kernel on $X$ with the multiplier separation property. Then, $\ell$ satisfies the automatic separation property.
\end{proposition}
\begin{proof}
Let $n\ge 3$ and suppose $\{\lambda_i\}\subset X$ is weakly separated by $\ell.$ Thus, there exists $\delta>0$ such that $d_s(\lambda_i, \lambda_j)>\delta$ for every $i\neq j$.  By assumption, there exists $\epsilon>0$ such that for every $i\neq j$, we can find $\phi_{ij}\in\text{Mult}(\mathcal{H}_{\ell})$ of norm at most $1$ such that $\phi_{ij}(\lambda_i)=\epsilon$ and $\phi_{ij}(\lambda_j)=0$. Now, suppose $\{\mu_1, \dots, \mu_n\}$ is an arbitrary $n$-point subset of $\{\lambda_i\}$. Consider the contractive multiplier $\Phi:=\prod_{i=2}^n\phi_{\mu_1\mu_i}$, which satisfies $\Phi(\mu_2)=\dots=\Phi(\mu_n)=0$ and $\Phi(\mu_1)=\epsilon^{n-1}$. But then, we will have $||\Phi\hat{\ell}_{\mu_1}||_{\mathcal{H}_{\ell}}\le 1$ and so we can write 
$$\frac{1}{\text{dist}\big(\hat{\ell}_{\mu_1}, \text{ span}\big\{\hat{\ell}_{\mu_2}, \dots, \hat{\ell}_{\mu_n}\big\} \big) }$$ $$=\inf\Big\{||f||: f\in\big(\text{ span}\big\{\hat{\ell}_{\mu_2}, \dots, \hat{\ell}_{\mu_n}\big\}\big)^{\perp}, \hspace{0.2 cm} \langle f, \hat{\ell}_{\mu_1} \rangle =1 \Big\} $$
$$=\inf\big\{||f||: f\in\mathcal{H}_{\ell}, \hspace{0.2 cm}  f(\mu_2)=\dots=f(\mu_n)=0, \hspace{0.2 cm}  f(\mu_1)=||\ell_{\mu_1}|| \big\} $$
$$\le \frac{||\Phi\hat{\ell}_{\mu_1}||}{\epsilon^{n-1}} $$
$$\le \frac{1}{\epsilon^{n-1}},$$
hence $\{\lambda_i\}$ must be $n$-weakly separated by $\ell$. This concludes the proof. 
\end{proof}
Suppose now that we have a kernel (or a collection of kernels) with the multiplier separation property. Our next result shows that performing certain operations on these kernels allows us to construct new ones possessing the same property. 

\begin{proposition} \label{105}
Suppose $\ell: X\times X\to\mathbb{C}$ and $k: S\times S\to\mathbb{C}$ are two kernels with the multiplier separation property and let $\phi: S\to X$ be a function and $a\ge 1$. Suppose also that $\rho$ denotes any of the following kernels:
\begin{itemize}
    \item[(a)] $\ell\otimes k$;
    \item[(b)] $\ell^a$ (here, we also assume that $\ell$ is non-vanishing);
     \item[(c)] $\ell\circ \phi$ (defined by $\ell\circ \phi(\lambda, \mu)=\ell(\phi(\lambda), \phi(\mu))$);
    \item[(d)] a kernel $\rho:X\times X\to\mathbb{C}$ such that both $\rho/\ell$ and $\ell^a/\rho$ are positive semi-definite;
    \item[(e)]  a kernel $\rho: X\times X\to\mathbb{C}$ such that $\ell/\rho$ is positive semi-definite and $||\phi||_{\text{Mult}(\mathcal{H}_{\rho})}\le C \sup_{x\in X}|\phi(x)|$, for some constant $C\ge 1$, for every $\phi$ in Mult$(\mathcal{H}_{\rho})$;
    \item[(f)] a kernel $\rho: X\times X\to\mathbb{C}$ such that $\mathcal{H}_{\rho}$ and $\mathcal{H}_{\ell}$ have equivalent norms.
    
\end{itemize}
  Then, $\rho$ must also have the multiplier separation property.
\end{proposition}

Before we go into the proof, we require the following simple lemma.
\begin{lemma} \label{104}
Suppose $g, \ell$ are two reproducing kernels on $X$ such that $\ell/g$ is positive semi-definite. Then, $\text{Mult}(\mathcal{H}_{g})\subset \text{Mult}(\mathcal{H}_{\ell})$ and $  ||\phi||_{\text{Mult}(\mathcal{H}_{\ell})}\le  ||\phi||_{\text{Mult}(\mathcal{H}_{g})}$, for every $\phi\in\text{Mult}(\mathcal{H}_{g})$. 

\end{lemma}
\begin{proof}
We know that $||\phi||_{\text{Mult}(\mathcal{H}_{g})}\le M$ if and only if the matrix $\big[\big(M^2-\phi(\lambda_i)\overline{\phi(\lambda_j)}\big)g(\lambda_i, \lambda_j)  \big]$
is positive semi-definite for any choice of points $\lambda_i$ in $X$. Taking the Schur product with $\big[(\ell/g)(\lambda_i, \lambda_j)\big]$ then yields the desired result.
\end{proof}

\begin{proof}[Proof of Proposition \ref{105}]
We only prove parts (a) and (c)-(f) (the proof of (b) is similar to that of (a)). \\
For (a), let $\delta>0$ and suppose $(\lambda_i, \mu_i), (\lambda_j, \mu_j)$ are two points in $X\times S$ satisfying $d_{\ell\otimes k}\big((\lambda_i, \mu_i), (\lambda_j, \mu_j)\big)>\delta$. Thus, we must have 
$$\frac{| (\ell\otimes k)\big((\lambda_i, \mu_i), (\lambda_j, \mu_j)\big)|^2}{(\ell\otimes k)\big((\lambda_i, \mu_i), (\lambda_i, \mu_i)\big)(\ell\otimes k)\big((\lambda_j, \mu_j), (\lambda_j, \mu_j)\big)}$$ $$=\frac{|\ell(\lambda_i, \lambda_j)k(\mu_i, \mu_j) |^2}{\ell(\lambda_i, \lambda_i)\ell(\lambda_j, \lambda_j)k(\mu_i, \mu_i)k(\mu_j, \mu_j)}$$ $$\le 1-\delta^2,$$
which implies that either $|\langle \hat{\ell}_{\lambda_i}, \hat{\ell}_{\lambda_j} \rangle|^2\le \sqrt{1-\delta^2}$ or $|\langle \hat{k}_{\mu_i}, \hat{k}_{\mu_j} \rangle|^2\le \sqrt{1-\delta^2}$. Without loss of generality, assume that $|\langle \hat{\ell}_{\lambda_i}, \hat{\ell}_{\lambda_j} \rangle|^2\le \sqrt{1-\delta^2}$, hence $$d_{\ell}(\lambda_i, \lambda_j)\ge \sqrt{1-\sqrt{1-\delta^2}}.$$ The fact that $\ell$ has the multiplier separation property then implies the existence of $\epsilon>0$ (depending only on $\delta$) and $\phi_{ij}\in\text{Mult}(\mathcal{H}_{\ell})$ of norm at most $1$ such that $\phi_{ij}(\lambda_i)=\epsilon$ and $\phi_{ij}(\lambda_j)=0.$ In view of (\ref{706}), we obtain
\begin{equation}\label{709}
  (1-\phi_{ij}(x)\overline{\phi_{ij}(y)})\ell(x, y)\ge 0,
\end{equation} 
for every choice of points $x, y\in X.$ We now extend $\phi$ to $X\times S$ by putting $\phi(x, s)=\phi(x)$, for every $(x, s)\in X\times S.$ Condition (\ref{709}) then becomes 
$(1-\phi_{ij}(x, s)\overline{\phi_{ij}(y, t)})\ell(x, y)\ge 0$, which, after taking the Schur product with $k(s, t)$, gives us 
$$(1-\phi_{ij}(x, s)\overline{\phi_{ij}(y, t)})\ell(x, y)k(s, t)$$ $$=(1-\phi_{ij}(x, s)\overline{\phi_{ij}(y, t)})(\ell\otimes k)\big((x, s), (y, t)\big)\ge 0,$$ for every choice of points $(x, s), (y, t)\in X\times S.$ Hence, $\phi_{ij}$ is a contractive multiplier of $\mathcal{H}_{\ell\otimes k},$ which concludes the proof of (a).
\par
For (c), let $\delta>0$ and assume that $s_i, s_j$ are two points in $S$ satisfying $d_{\ell\circ\phi}(s_i, s_j)>\delta$. It is known (see \cite[Theorem 5.7]{PauRa}) that there exists an isometry $\Gamma: \mathcal{H}_{\ell\circ\phi}\to\mathcal{H}_{\ell}$ satisfying $\Gamma((\ell\circ\phi)_s)=\ell_{\phi(s)},$ for every $s\in S.$ This implies that $d_{\ell\circ\phi}(s, t)=d_{\ell}(\phi(s), \phi(t)),$ for every $s, t\in S.$ Thus, we obtain $d_{\ell}(\phi(s_i), \phi(s_j))>\delta.$ But $\ell$ has the multiplier separation property, so we deduce the existence of $\epsilon>0$ (depending only on $\delta$) and $\psi_{ij}\in\text{Mult}(\mathcal{H}_{\ell})$ of norm at most $1$ such that $\psi_{ij}(\phi(s_i))=\epsilon$ and $\psi_{ij}(\phi(s_j))=0.$ Since (again by \cite[Theorem 5.7]{PauRa}) $||f||_{\mathcal{H}_{\ell\circ\phi}}=\inf\{||F||_{\mathcal{H}_{\ell}}: f=F\circ\phi \}$, for every $f\in \mathcal{H}_{\ell\circ\phi},$ we obtain $||\psi_{ij}\circ \phi||_{\text{Mult}(\mathcal{H}_{\ell\circ\phi})}\le ||\psi_{ij}||_{\text{Mult}(\mathcal{H}_{\ell})}\le 1$ and so $\psi_{ij}\circ\phi$ is a separating multiplier with the desired properties.
\par

For (d), suppose that $\rho$ is a kernel on $X$ satisfying the given hypotheses. Let $\delta>0$ and assume that $\lambda_i, \lambda_j$ are two points in $X$ with $d_{\rho}(\lambda_i, \lambda_j)>\delta$. By Sylvester's criterion, we obtain the positivity of the $2\times 2$ matrix 
$$
\begin{bmatrix}
(1-\delta^2)\rho(\lambda_i, \lambda_i) & \rho(\lambda_i, \lambda_j) \\
\rho(\lambda_j, \lambda_i) & \rho(\lambda_j, \lambda_j). 
\end{bmatrix}
$$
Taking the Schur product with the positive $2\times 2$ matrix $\big[(\ell^a/\rho)(\lambda_i, \lambda_j)\big]$ then yields the positivity of 
 $$ \begin{bmatrix}
(1-\delta^2)\ell^a(\lambda_i, \lambda_i) & \ell^a(\lambda_i, \lambda_j) \\
\ell^a(\lambda_j, \lambda_i) & \ell^a(\lambda_j, \lambda_j)
\end{bmatrix}, $$
which implies that $d_{\ell^a}(\lambda_i, \lambda_j)>\delta$. Thus, we obtain $d_{\ell}(\lambda_i, \lambda_j)>\sqrt{1-\sqrt[a]{1-\delta^2}}$. But $\ell$ has the multiplier separation property, so we deduce the existence of $\epsilon>0$ (depending only on $\delta$ and $a$) and $\phi_{ij}\in\text{Mult}(\mathcal{H}_{\ell})$ of norm at most $1$ such that $\phi_{ij}(\lambda_i)=\epsilon$ and $\phi_{ij}(\lambda_j)=0.$  By Lemma \ref{104}, $\phi_{ij}$ must also be a contractive multiplier of $\mathcal{H}_{\rho},$ which concludes the proof.

For (e), again suppose that $\rho$ is a kernel on $X$ satisfying the given hypotheses and let $\lambda_i, \lambda_j$ be two points in $X$ satisfying $d_{\rho}(\lambda_i, \lambda_j)>\delta$. Reasoning as in the proof of (d), we obtain the existence of $\epsilon>0$ (depending only on $\delta$) and $\phi_{ij}\in\text{Mult}(\mathcal{H}_{\ell})$ of norm at most $1$ such that $\phi_{ij}(\lambda_i)=\epsilon$ and $\phi_{ij}(\lambda_j)=0.$ But  $\mathcal{H}_{\ell}$ is a Hilbert function space, so, in view of our assumptions, we can write
$$||\phi_{ij}||_{\text{Mult}(\mathcal{H}_{\rho})}\le C\sup_{x\in X}|\phi_{ij}(x)|\le C||\phi_{ij}||_{\text{Mult}(\mathcal{H}_{\ell})}\le C,$$
which implies that $\phi_{ij}$ is a separating multiplier with the desired properties.

\par Finally, suppose that $\mathcal{H}_{\ell}$ and $\mathcal{H}_{\rho}$ have equivalent norms. This implies that the evaluation functionals $T_x:\mathcal{H}_{\ell}\to\mathbb{C}, \text{ } T_x(f)=f(x)$ and $S_x:\mathcal{H}_{\rho}\to\mathbb{C}, \text{ } S_x(f)=f(x)$ also have equivalent norms (with constants independent of $x\in  X$). Hence, there exist $C_1, C_2>0$ such that \begin{equation}\label{812} C_1||\ell_x||_{\mathcal{H}_{\ell}}\le ||\rho_x||_{\mathcal{H}_{\rho}}\le C_2  ||\ell_x||_{\mathcal{H}_{\ell}},\end{equation}  for every $x\in X.$ But since we know that, for any kernel $k$ and points $x,y \in X$, 
$$ \frac{1}{d_{k}(x, y)}=\frac{1}{\text{dist}\big(\hat{k}_{x}, \text{ span}\big\{\hat{k}_{y}\big\} \big) }$$ $$=\inf\Big\{||f||_{\mathcal{H}_k}: f\in\big(\text{ span}\big\{\hat{k}_{y}\big\}\big)^{\perp}, \hspace{0.2 cm} \langle f, \hat{k}_{x} \rangle =1 \Big\} $$
\begin{equation} \label{1303}=\inf\big\{||f||_{\mathcal{H}_k}: f\in\mathcal{H}_{k}, \hspace{0.2 cm}  f(y)=0, \hspace{0.2 cm}  f(x)=||k_{x}||_{\mathcal{H}_k} \big\}, \end{equation}
(\ref{812}) and the equivalence of $||\cdot||_{\mathcal{H}_{\ell}}$ and $||\cdot||_{\mathcal{H}_{\rho}}$ allow us to deduce the existence of constants $C'_1, C'_2>0$ such that $$C'_1d_{\ell}(x, y)\le d_{\rho}(x, y)\le C'_2d_{\ell}(x, y),$$
for all $x, y\in X$. This double inequality, combined with the fact that the associated multiplier norms $||\cdot||_{\text{Mult}(\mathcal{H}_{\ell})}$ and $||\cdot||_{\text{Mult}(\mathcal{H}_{\rho})}$ must also be equivalent, finishes off the proof. 
\end{proof}

\begin{remark}
Let $\ell, k$ be two kernels on $X$ with the multiplier separation property. In view of Proposition \ref{105}, the product $\ell\cdot k$ (which is the restriction of $\ell\otimes k$ along the diagonal) must also have the multiplier separation property.
\end{remark}
\begin{remark}
Assume that $\ell, \rho, k$ are kernels on $X$ such that $\ell$ and $k$ have the multiplier separation property and $\rho/\ell$ and $k/\rho$ are both positive semi-definite. Then, $\rho$ needn't even be an AS kernel. Indeed, define 
$$\rho(\lambda, \mu)=\frac{1}{1-\lambda^2\overline{\mu}^2}+\frac{1}{1-\lambda^3\overline{\mu}^3}, \hspace{0.3 cm}\ell(\lambda, \mu)=\frac{1}{1-\lambda^6\overline{\mu}^6},$$
and $$k(\lambda, \mu)=\frac{1}{(1-\lambda^2\overline{\mu}^2)(1-\lambda^3\overline{\mu}^3)},$$
for $\lambda, \mu\in \DD.$ Now, since $\frac{1}{1-\lambda^2\overline{\mu}^2}, \frac{1}{1-\lambda^3\overline{\mu}^3}$ are both factors of $\frac{1}{1-\lambda\overline{\mu}}$, \cite[Proposition 5.6]{PauRa} implies that $\ell$ is a factor of both $\frac{1}{1-\lambda^2\overline{\mu}^2}$ and $\frac{1}{1-\lambda^3\overline{\mu}^3}$. Thus, $\rho/\ell$ is positive semi-definite. Also, since $\ell$ is a CP kernel, it evidently possesses the multiplier separation property. On the other hand, note that $$\frac{k(\lambda, \mu)}{\rho(\lambda, \mu)}=\frac{1/2}{1-\langle \frac{1}{\sqrt{2}} (\lambda^2, \lambda^3), \frac{1}{\sqrt{2}} (\mu^2, \mu^3) \rangle},\hspace{0.3 cm} \lambda, \mu\in\DD,$$
and so $k/\rho$ is positive semi-definite (actually a CP kernel). Being the product of two CP kernels, $k$ must satisfy the multiplier separation property. However, note that (letting $\omega=e^{2\pi i/3}$) $$\rho_{z}-\rho_{\omega z}+\rho_{-\omega z}-\rho_{-z}=0,$$
for all $z\in\DD$. Also, any two-vector subset of $\{\rho_{z}, \rho_{\omega z}, \rho_{-\omega z}, \rho_{-z} \}$ is linearly independent if $z\neq 0.$ This implies that $\rho$ does not have the AS property.
\end{remark}

We now present examples of kernels satisfying the multiplier separation property. 
\begin{exmp}[\textit{Products of powers of $2$-point Pick kernels}] \label{821} \hfill \par
Let $k=k_1^{t_1}\otimes k_2^{t_2}\otimes\cdots\otimes k_n^{t_n}$, where $n\ge 1,$ $t_i\ge 1$ and each $k_i$ is an irreducible kernel on $X_i$ with the $2$-point scalar Pick property (note that, in view of \cite[Lemma 7.2]{Pick}, every such kernel must be nonzero on $X_i\times X_i$). If $x, y\in X_i$ satisfy $d_{k_i}(x, y)=\delta,$ then the $2$-point Pick property implies the existence of a contractive multiplier $\phi_{xy}\in\text{Mult}(\mathcal{H}_{k_i})$ such that $\phi_{xy}(x)=\delta$ and $\phi_{xy}(y)=0.$ Thus, each $k_i$ has the multiplier separation property and we also deduce, in view of Proposition \ref{105}, that the same must be true for $k$. Examples of such kernels (which are actually products of powers of complete Pick kernels) include those of the form
$$k((\lambda_1, \dots, \lambda_m), (\mu_1, \dots, \mu_m))=\prod_{i=1}^m\frac{1}{(1-\langle b_i(\lambda_i), b_i(\mu_i)\rangle)^{t_i}},$$
where $t_i\ge 1$, $b_i: X_i\to\mathbb{B}_d$ and  $(\lambda_1, \dots, \lambda_m), (\mu_1, \dots, \mu_m)$ lie in the polydomain $X_1\times X_2\times \cdots \times X_m$. 
\par If each $k_i$ is defined on the same set $X,$  Proposition \ref{105}(c) and the previous result imply that $\tilde{k}:=k_1^{t_1}k_2^{t_2}\cdots k_n^{t_n}$ must also have the automatic separation property whenever every $k_i$ is an irreducible $2$-point Pick kernel.
\end{exmp}

\begin{exmp}[\textit{Hardy spaces on planar domains}] \label{822}  \hfill \par
Suppose that $\Omega\subset\mathbb{C}$ is a domain with boundary $\partial\Omega$ consisting of a finite collection of smooth curves. Let $d\sigma$ be arclength measure on $\partial\Omega$ and define $H^2(\Omega)$ to be the closure in $L^2(\partial\Omega, d\sigma)$ of the subspace consisting of restrictions to $\partial\Omega$ of functions holomorphic on $\overline{\Omega}$ (see \cite{Abrahamse} and \cite{Duren} for the basic theory of these spaces). While the choice of the measure $d\sigma$ is not canonical, all the standard choices lead to the same space of holomorphic functions
on $\Omega$ with equivalent norms.  A fascinating result due to Arcozzi, Rochberg and Sawyer (see \cite[Corollary 13]{Rochberg}) then tells us that $H^2(\Omega)$ admits an equivalent norm with the property that with the new norm the space is a reproducing kernel Hilbert space with a complete Pick kernel. In view of Proposition \ref{105}(f), we obtain that the kernel of $H^2(\Omega)$ has the multiplier separation property.

\end{exmp}

 Examples \ref{821}-\ref{822} serve as manifestations of a general observation: the multiplier separation property will always be present in kernels obtained by performing any of the operations from Proposition \ref{105} to one or more $2$-point Pick kernels. We state this as a corollary.
\begin{corollary} \label{843}
Let $\ell$ be a kernel on $X$ defined by performing a finite number of any of the operations from Proposition \ref{105} to one or more kernels having the $2$-point Pick property. Then, $\ell$ has the multiplier separation property. 
\end{corollary}

\subsection{The general case} \label{303}
First, we prove a result (which mirrors Proposition \ref{105}) showing that certain operations on kernels preserve the AS property.
\begin{proposition} \label{942}
Suppose $\ell: X\times X\to\mathbb{C}$ and $k: S\times S\to\mathbb{C}$ are two AS kernels and let $\phi: S\to X$ be a function and $a\ge 1$. Suppose also that $\rho$ denotes any of the following kernels:
\begin{itemize}
    \item[(a)] $\ell\otimes k$;
    \item[(b)] $\ell^a$ (here, we also assume that $\ell$ is non-vanishing);
     \item[(c)] $\ell\circ \phi$ (defined by $\ell\circ \phi(\lambda, \mu)=\ell(\phi(\lambda), \phi(\mu))$);
    \item[(d)] a kernel $\rho:X\times X\to\mathbb{C}$ such that both $\rho/\ell$ and $\ell^a/\rho$ are positive semi-definite;
    \item[(e)] a kernel $\rho: X\times X\to\mathbb{C}$ such that $\mathcal{H}_{\rho}$ and $\mathcal{H}_{\ell}$ have equivalent norms.
    
\end{itemize}
  Then, $\rho$ must also have the AS property.
\end{proposition}

\begin{proof}
The ideas here are very similar to those used in the proof of Proposition \ref{105} (one difference being that we have to use Lemma \ref{944} in place of Lemma \ref{104}), so we only prove (c)-(e). \\
For (c), let $\{s_i\}\subset S$ be weakly separated by $\ell\circ\phi$, hence we can find $\epsilon>0$ such that any two points $s_i\neq s_j$ satisfy $d_{\ell\circ\phi}(s_i, s_j)>\epsilon$. Recall (as in the proof of Proposition \ref{105}(c)) that 
\begin{equation} \label{920}
    \langle (\ell\circ\phi)_{s}, (\ell\circ\phi)_{t} \rangle_{\mathcal{H}_{\ell\circ\phi}}= \langle \ell_{\phi(s)}, \ell_{\phi(t)} \rangle_{\mathcal{H}_{\ell}},
\end{equation}
for every $s, t\in S$. This implies that $d_{\ell}(\phi(s_i), \phi(s_j))=d_{\ell\circ\phi}(s_i, s_j)>\epsilon,$ for every $i\neq j$, hence the sequence $\{\phi(s_i)\}\subset X$ is weakly separated by $\ell$. But $\ell$ has the AS property, so $\{\phi(s_i)\}$ must also be $n$-weakly separated by $\ell,$ for every $n\ge 3.$ Thus,   we can find positive constants $\epsilon_n>0$ such that for every $n$-point subset $\{\phi(\mu_1), \dots, \phi(\mu_n)\}$ of $\{\phi(s_i)\}$ and for $w_1=\epsilon_n$, $w_2=\dots=w_n=0$, the matrix
$$[(1-w_j\overline{w}_i)\langle \hat{\ell}_{\phi(\mu_i)},  \hat{\ell}_{\phi(\mu_j)} \rangle_{\mathcal{H}_{\ell}}]_{1\le i,j\le n}$$
is positive semi-definite. In view of (\ref{920}), the same must be true for the matrix 
$$[(1-w_j\overline{w}_i)\langle \widehat{(\ell\circ\phi)}_{\mu_i},  \widehat{(\ell\circ\phi)}_{\mu_j} \rangle_{\mathcal{H}_{\ell\circ\phi}}]_{1\le i,j\le n}.$$
Lemma \ref{708} then implies that $\{s_i\}$ is $n$-weakly separated by $\ell\circ\phi,$ for every $n\ge 3,$ so the proof of (c) is complete. \par 
For (d), suppose that $\rho$ is a kernel on $X$ satisfying the given hypotheses. Let $\{\lambda_i\}\subset X$ be a sequence satisfying $d_{\rho}(\lambda_i, \lambda_j)>\epsilon>0$, for every $i\neq j$. As in the proof of Proposition \ref{105}(d), we obtain $d_{\ell}(\lambda_i, \lambda_j)>\sqrt{1-\sqrt[a]{1-\epsilon^2}}$, for every $i\neq j$. This implies that $\{\lambda_i\}$ is weakly, and hence $n$-weakly, separated by $\ell$, for every $n\ge 3$. Lemma \ref{944} then allows us to deduce that $\{\lambda_i\}$ must also be $n$-weakly separated by $\rho,$ for every $n\ge 3,$ which concludes the proof. \par

Finally, suppose that $\mathcal{H}_{\ell}$ and $\mathcal{H}_{\rho}$ have equivalent norms. As in the proof of Proposition \ref{105} (f), there exist constants $C_1, C_2>0$ such that (\ref{812}) is satisfied. But then, we also know that for any kernel $k: X\times X\to\mathbb{C}$ and any $n$-point set $\{\mu_1, \dots, \mu_n\}\subset X$ we can write
$$\frac{1}{\text{dist}\big(\hat{k}_{\mu_1}, \text{ span}\big\{\hat{k}_{\mu_2}, \dots, \hat{k}_{\mu_n}\big\} \big) }$$ 
$$=\inf\big\{||f||: f\in\mathcal{H}_{k}, \hspace{0.2 cm}  f(\mu_2)=\dots=f(\mu_n)=0, \hspace{0.2 cm}  f(\mu_1)=||k_{\mu_1}|| \big\}. $$
This equality, combined with (\ref{812}) and the equivalence of norms for $\mathcal{H}_{\ell}, \mathcal{H}_{\rho}$, implies that a sequence in $X$ is $n$-weakly separated by $\ell $ if and only if it is $n$-weakly separated by $\rho,$ for any $n\ge 2,$ and so, since $\ell$ is an AS kernel, we are done.
\end{proof}

\begin{remark}
Let $\ell, k$ be two AS kernels on $X$. In view of Proposition \ref{942}, the product $\ell\cdot k$ must have the AS property as well.
\end{remark}

Next, we establish a general criterion for the AS property, one that is closely related to the nature of interpolating sequences for $\mathcal{H}_{\ell}$. In particular, we will show that, under some mild additional assumptions, the kernel $\ell$ satisfies the AS property if and only if any weakly separated finite union of ``sufficiently sparse" sequences in $X$ forms an $\mathcal{H}_{\ell}$-interpolating sequence.

\begin{theorem}\label{833}
Suppose $X$ is a topological space, $\ell$ is a kernel on X and the following properties are satisfied:
\begin{itemize}
    \item[(Q0)] No finite collection of kernel functions $\ell_{\lambda_1}, \dots, \ell_{\lambda_m}$ can form a linearly dependent set if $|\{\lambda_1,\dots, \lambda_m\}|=m$; 
    \item[(Q1)] $\ell: X\times X\to\mathbb{C}$ is continuous;
      \item[(Q2)] If $\{\lambda_i\}\subset X$ satisfies $||\ell_{\lambda_i}||\to\infty$, then $||\ell_{\lambda_i}||^{-1}\ell(\lambda_i, \mu)\to 0$ for every $\mu\in X$;
      \item[(Q3)] Let $\{\lambda_i\}\subset X$. Then,  either $||\ell_{\lambda_i} ||\to\infty$ or $\{\lambda_i\}$ contains a subsequence converging to a point inside $X$.
    
\end{itemize}
Then, the following assertions are equivalent:
\begin{itemize}
    \item[(i)] $\ell$ has the automatic separation property.
    \item[(ii)] Let $n\ge 3$ and suppose $\{\lambda_i\}$ is a weakly separated by $\ell$ sequence that can be written as $\{\lambda_i\}=\cup^{\infty}_{j=1}\cup^n_{k=1}\{\mu^k_j\}$, where $\big|\big|\ell_{\mu^k_j}\big|\big|\to\infty$ as $j\to\infty,$ for every $1\le k\le n.$ Then, we can always find a subsequence $\{m_i\}$ such that $\cup^{\infty}_{j=1}\cup^n_{k=1}\{\mu^k_{m_j}\}$ is interpolating for $\mathcal{H}_{\ell}$.

\end{itemize}

\end{theorem}
Note that the extra hypotheses imposed on $\ell$ include properties (Q1)-(Q3) from the statement of Theorem \ref{89}, as well as a new condition (Q0) ensuring that no finite collection of kernel functions can be linearly dependent. This is to avoid the somewhat trivial situation where the failure of the AS property would be caused by the existence of linearly dependent kernel functions $\ell_{\lambda_1}, \dots, \ell_{\lambda_m}$ (where $m\ge 3$) such that no two kernels $\ell_{\lambda_i}, \ell_{\lambda_j}$ are linearly dependent if $i\neq j$.

\begin{proof}
First, suppose that $\ell$ satisfies the hypotheses of (ii). Working towards a contradiction, assume that there exists $\{\lambda_i\}\subset X$ and $n\ge 3$ such that $\{\lambda_i\}$ is $(n-1)$-weakly separated but not $n$-weakly separated by $\ell$. Thus, we can find a subsequence $\cup^{\infty}_{j=1}\cup^n_{k=1}\{\mu^k_j\}$ of $\{\lambda_i\}$ where $|\{\mu^1_j, \dots, \mu^n_j\}|=n$ for every $j$ (but we may have $\mu^k_j=\mu^r_i$ if $k\neq r$) and 
\begin{equation}
    \lim_j\text{dist}\big(\hat{\ell}_{\mu^1_j}, \text{ span}\big\{\hat{\ell}_{\mu^2_j}, \dots, \hat{\ell}_{\mu^n_j}\big\} \big)=0.
    \label{113}
\end{equation}
In view of Q(0), $\cup^{\infty}_{j=1}\cup^n_{k=1}\{\mu^k_j\}$ must contain infinitely many points. We can also assume, without loss of generality, that each subsequence $\{\mu^k_j\}_j$ either satisfies $\big|\big|\ell_{\mu^k_j}\big|\big|\to\infty$ or converges to a point in $X.$ This is possible because of (Q3) (we can keep extracting subsequences until we have the desired properties). \par

Now, if every $\{\mu^k_j\}_j$ converges to a point $p_k$ in $X$, we
can show (as in the proof of Theorem \ref{89}) that the kernel functions $\ell_{p_1}, \dots, \ell_{p_n}$ are linearly dependent, which contradicts (Q0). On the other hand, if at least one subsequence $\{\mu^{k'}_j\}_j$ satisfies $\big|\big|\ell_{\mu^{k'}_j}\big|\big|\to\infty$ and at least one subsequence $\{\mu^k_j\}_j$ converges to a point in $X$, we
arrive at a contradiction (using properties (Q1)-(Q3)) in precisely the same manner as in the proof of Theorem \ref{89}. Finally, if $\big|\big|\ell_{\mu^k_j}\big|\big|\to\infty$ as $j\to\infty, $ for every $k$, the sequence $\cup^{\infty}_{j=1}\cup^n_{k=1}\{\mu^k_j\}$ will satisfy the hypotheses of (ii). Thus, we can extract a subsequence $\cup^{\infty}_{j=1}\cup^n_{k=1}\{\mu^k_{m_j}\}$ of $\cup^{\infty}_{j=1}\cup^n_{k=1}\{\mu^k_j\}$ that is interpolating for $\mathcal{H}_{\ell}.$ In particular, $\cup^{\infty}_{j=1}\cup^n_{k=1}\{\mu^k_{m_j}\}$ must be $n$-weakly separated by $\ell$, which contradicts (\ref{113}). Hence, $\ell$ must be an AS kernel. \par
For the converse, assume that $\ell$ possesses the automatic separation property. Let $n\ge 3$ and suppose $\{\lambda_i\}$ is a weakly separated by $\ell$ sequence that can be written as $\{\lambda_i\}=\cup^{\infty}_{j=1}\cup^n_{k=1}\{\mu^k_j\}$, where $\big|\big|\ell_{\mu^k_j}\big|\big|\to\infty$ for every $1\le k\le n.$ Assume also, for convenience, that $|\{\mu^1_j, \dots, \mu^n_j \}|=n$ for every $j$. The AS property tells us that  $\{\lambda_i\}$ has to be $n$-weakly separated by $\ell$, for every $n\ge 3$. Arguing as in the proof of Lemma \ref{18}, we can show that there exists $\epsilon>0$ such that the  matrices 
$$\Big[(1-w_{m, k}\overline{w}_{m. r})\big\langle \hat{\ell}_{\mu^r_j}, \hat{\ell}_{\mu^k_j}  \big\rangle \Big]_{1\le k,r \le n} $$
are positive semi-definite for every $m\in\{1, \dots, n\}$ and $j\ge 1$, where $w_{k, k}=\epsilon$ and $w_{m, k}=0$ if $m\neq k$ (without the assumption that $|\{\mu^1_j, \dots, \mu^n_j \}|=n$, we would need to consider matrix blocks of non-constant size $|\{\mu^1_j, \dots, \mu^n_j \}|\le n$, but that wouldn't affect the proof in any way). Adding together the matrices corresponding to $m=1, 2,\dots, n$, we obtain the positivity condition
\begin{equation}
    \sum_{k, r=1}^n a^k\overline{a^r}\big\langle \hat{\ell}_{\mu^r_j}, \hat{\ell}_{\mu^k_j}  \big\rangle \ge c\sum_{k=1}^n|a^k|^2, \label{121}
\end{equation}
for every $j\ge 1$ and every choice of scalars $a^k, a^r$ ($c$ can be taken to be $\epsilon^2/n$). 
\\ Our next step will be to extract a subsequence  $\cup^{\infty}_{j=1}\cup^n_{k=1}\{\mu^k_{m_j}\}$ that has a bounded below Grammian. To achieve this, it suffices to find a subsequence $\{m_j\}$ such that for every $\rho\ge 1,$
\begin{equation}
  \sum^{\rho}_{i, j=1}\sum^n_{k, r=1}a^k_i\overline{a^r_j} \Big\langle \hat{\ell}_{\mu^r_{m_j}}, \hat{\ell}_{\mu^k_{m_i}} \Big\rangle > \frac{c}{2}\bigg[2-\sum_{i=1}^{\rho}2^{-i} \bigg]\sum^{\rho}_{i=1}\sum^n_{k=1} |a^k_i|^2,  
 \label{122}   
\end{equation}
for every choice of scalars $a^k_i, a^r_j$. \par We will construct $\{m_j\}$ inductively. First, choose $m_1=1$. Then, (\ref{122}) will be satisfied (for $\rho=1$) because of (\ref{121}). Now, suppose that $\rho$ integers $\{m_1, \dots, m_{\rho}\}$ have been chosen so that (\ref{122}) is satisfied. Without loss of generality, we can take our scalars $a^k_i$ to satisfy $\sum^{\rho+1}_{i=1}\sum^n_{k=1} |a^k_i|^2= 1$. By assumption, we know that $\big|\big|\ell_{\mu^k_j}\big|\big|\to\infty$ for every $k$ and hence, by Lemma (\ref{85}), there exists an integer $m_{\rho+1}$ such that 
\begin{equation}
    \big|\big\langle  \hat{\ell}_{\mu^r_{m_{\rho+1}}}, \hat{\ell}_{\mu^k_{m_i}}  \big\rangle\big|\le \frac{2^{-(\rho+2)}}{2n^2\rho},
 \end{equation}
for every $k, r\in\{1, \dots,n\}$ and $i\in\{1, \dots, \rho\}$. This implies that 
\begin{equation}
2\Bigg|\sum^{\rho}_{i=1}\sum^n_{k, r=1}a^k_i\overline{a^r_{\rho+1}} \Big\langle \hat{\ell}_{\mu^r_{m_{\rho+1}}}, \hat{\ell}_{\mu^k_{m_i}} \Big\rangle\Bigg| \le 2^{-(\rho+2)}c.
 \label{123}
\end{equation}
Now, we can combine our inductive hypothesis with (\ref{121}) and (\ref{123})  to obtain 
$$ \sum^{\rho+1}_{i, j=1}\sum^n_{k, r=1}a^k_i\overline{a^r_j} \Big\langle \hat{\ell}_{\mu^r_{m_j}}, \hat{\ell}_{\mu^k_{m_i}} \Big\rangle $$ $$= \sum^{\rho}_{i, j=1}\sum^n_{k, r=1}a^k_i\overline{a^r_j} \Big\langle \hat{\ell}_{\mu^r_{m_j}}, \hat{\ell}_{\mu^k_{m_i}} \Big\rangle+2\Re\Bigg[ \sum^{\rho}_{i=1}\sum^n_{k, r=1}a^k_i\overline{a^r_{\rho+1}} \Big\langle \hat{\ell}_{\mu^r_{m_{\rho+1}}}, \hat{\ell}_{\mu^k_{m_i}} \Big\rangle\Bigg]$$ $$+  \sum_{k, r=1}^n a^k_{\rho+1}\overline{a^r_{\rho+1}}\Big\langle \hat{\ell}_{\mu^r_{m_{\rho+1}}}, \hat{\ell}_{\mu^k_{m_{\rho+1}}}  \Big\rangle$$ $$> \frac{c}{2}\bigg[2-\sum_{i=1}^{\rho}2^{-i} \bigg]\sum^{\rho}_{i=1}\sum^n_{k=1} |a^k_i|^2-2^{-(\rho+2)}c+c\sum^n_{k=1}|a^k_{\rho+1}|^2>$$
$$>\frac{c}{2}\bigg[2-\sum_{i=1}^{\rho}2^{-i} \bigg]\sum^{\rho+1}_{i=1}\sum^n_{k=1} |a^k_i|^2-2^{-(\rho+2)}c $$
$$=\frac{c}{2}\bigg[2-\sum_{i=1}^{\rho+1}2^{-i} \bigg],$$
as desired. \par 
We have thus found a subsequence $\cup^{\infty}_{j=1}\cup^n_{k=1}\{\mu^k_{m_j}\}$ that has a bounded below Grammian. Now, since (by Lemma (\ref{85})) $\hat{\ell}_{\mu^k_{m_j}}$ converges to $0$ weakly for every $k\in\{1, \dots, n\}$, we easily see that it must have a subsequence (which we denote by $\big\{\hat{\ell}_{\nu^k_{j}}\big\}$) that is a Riesz sequence. This implies that $\cup^{\infty}_{j=1}\cup^n_{k=1}\{\nu^k_{j}\}$ must have a bounded Grammian and so it has to be interpolating for $\mathcal{H}_{\ell}$ by Lemma \ref{131}. This concludes the proof.
\end{proof}

 An apparent limitation in applying Theorem \ref{833} is that one first needs to know  that there exists a sufficient condition for $\mathcal{H}_{\ell}$-interpolation having the form \{(WS) by $\ell$ \text{ }$+$\text{ } (D)\}, with the property that condition (D) is always satisfied by any finite union of ``sufficiently sparse" sequences. Of course, establishing the sufficiency of such a condition in the first place is usually a highly non-trivial problem! Nevertheless, as there already exists a large literature on the subject of interpolating sequences in different function spaces, Theorem \ref{833} will be of use in discovering new examples of AS kernels. \par 
 First, we prove a lemma which will aid us in situations where a different metric or distance function is used instead of the one induced by the kernel.
 \begin{lemma}\label{930}
 Let $\ell$ be a kernel on the topological space $X$  satisfying conditions (Q1)-(Q3) from the statement of Theorem \ref{833} and also let  $\rho: X\times X\to [0, \infty)$ be a continuous function such that $\rho(\lambda,  \mu)=0$ if and only if $\lambda=\mu$. Assume that if $\{\mu_i\}\subset X$ converges to a point in $X$ and $\{\nu_i\}\subset X$ satisfies $||\ell_{\nu_i}||\to\infty,$ then there exists $\epsilon>0$ and $i_0\in\mathbb{N}$ such that $$\rho(\mu_i, \nu_i)>\epsilon ,$$ for all $i\ge i_0$. Finally, suppose that there exists a condition (C) such that a sequence $\{\lambda_i\}\subset X$ is interpolating for $\mathcal{H}_{\ell}$ if and only if it satisfies (C) and also there exists $\delta>0$ such that $\rho(\lambda_i, \lambda_j)>\delta$, for every $i\neq j$. Then, a sequence in $X$ is interpolating for $\mathcal{H}_{\ell}$ if and only if it satisfies (C) and is weakly separated by $\ell.$
 \end{lemma}

\begin{proof}
 One direction is obvious; since any sequence $\{\lambda_i\}$ satisfying $\rho(\lambda_i, \lambda_j)>\delta$, for $i\neq j$, and (C) is interpolating for $\mathcal{H}_{\ell}$, it must also be weakly separated by $\ell.$ \par
 For the converse, suppose that $\{\lambda_i\}$ satisfies (C) and is weakly separated by $\ell.$ Aiming towards a contradiction, assume that there exist subsequences $\{\mu_i\}, \{\nu_i\}\subset \{\lambda_i\}$  such that $\mu_i\neq \nu_i$ for every $i$ and also 
 \begin{equation} \label{931}
 \lim_i\rho(\mu_i, \nu_i)=0.
 \end{equation}
 
 We can additionally assume (in view of (Q3)) that $\{\mu_i\}$ either converges to a point in $X$ or satisfies $||\ell_{\mu_i}||\to\infty$ and an analogous assumption can be made for $\{\nu_i\}$. Now, if both of them converge to points $p_1, p_2\in X$, the assumptions that $\ell$ is continuous and $\{\lambda_i\}$ is weakly separated imply that $d_{\ell}(p_1, p_2)\neq 0$, hence $p_1\neq p_2.$ But then, continuity of $\rho$ implies that 
 $$0= \lim_i\rho(\mu_i, \nu_i)=\rho(p_1, p_2)\neq 0,$$
 a contradiction. On the other hand, if $\{\mu_i\}$ converges to a point in $X$ and $\{\nu_i\}$ satisfies $||\ell_{\nu_i}||\to\infty$, we can find (in view of our assumptions) $\epsilon>0$ and $i_0\in\mathbb{N}$ such that $\rho(\mu_i, \nu_i)>\epsilon$ for all $i\ge i_0$, which again contradicts (\ref{931}). Finally, assume that $\{\mu_i\}, \{\nu_i\}$ both satisfy $||\ell_{\mu_i}||\to\infty$ and $||\ell_{\nu_i}||\to\infty$. Mimicking the proof of ``(i)$\Rightarrow$(ii)" from Theorem \ref{833}, we can extract a subsequence $\{m_i\}$ such that the union $\{\mu_{m_i}\}\cup \{\nu_{m_i}\}$ is interpolating for $\mathcal{H}_{\ell}$. This implies that $\{\mu_{m_i}\}\cup \{\nu_{m_i}\}$ must also satisfy $\rho(\mu_{m_i}, \nu_{m_i})>\delta>0$, for all $i$, which contradicts (\ref{931}). Our proof is complete.
 \end{proof}
 \begin{remark} \label{935}
In the setting of Lemma \ref{930}, assume that condition (C) possesses the following additional property:
 every sequence in $X$ that can be written as $\{\lambda_i\}=\cup^{\infty}_{j=1}\cup^n_{k=1}\{\mu^k_j\}$, where $\big|\big|\ell_{\mu^k_j}\big|\big|\to\infty$ as $j\to\infty,$ for every $1\le k\le n,$ contains a subsequence of the form $\cup^{\infty}_{j=1}\cup^n_{k=1}\{\mu^k_{m_j}\}$ that satisfies (C). If we also have that no two functions $\ell_z, \ell_w$ are linearly dependent if $z\neq w$, the conclusion of Lemma \ref{930} can be strengthened to give: $\{\lambda_i\}\subset X$ is weakly separated by $\ell$ if and only if there exists $\delta>0$ such that $\rho(\lambda_i, \lambda_j)>\delta$, for every $i\neq j$. We omit the proof of this claim, as it is very similar to that of the previous Lemma.
 \end{remark}
 
  Now, recall that, as seen in subsection \ref{302}, every kernel possessing the multiplier separation property must also be an AS kernel. In particular, Theorem \ref{843} tells us that, in a certain sense, ``proximity" to powers of products of $2$-point Pick kernels guarantees that the kernel will have the AS property. Well-studied examples of spaces associated with such kernels include (apart, of course, from complete Pick spaces) Bergman spaces with standard weights on the $n$-dimensional unit ball, where the reproducing kernel is equal to
  $$\ell_{\alpha}(\mathbf{z}, \mathbf{w})=\frac{1}{(1-\langle \mathbf{z}, \mathbf{w} \rangle)^{n+1+\alpha}}, \hspace{0.2 cm} \mathbf{z}, \mathbf{w}\in\mathbb{B}_n,  \hspace{0.1 cm} \alpha>-1.$$
  The norm of $\mathcal{H}_{\ell_{\alpha}}$ is given by integration against the weighted Lebesgue measure $dv_{\alpha}$ defined as
  $dv_{\alpha}(z)=c_{\alpha}(1-|z|^2)^{\alpha}dv(z),$ where $c_{\alpha}$ ensures that $dv_{\alpha}$ is a probability measure. See \cite{Zhuspaces} for more details on these spaces.
  
  \par 
  
  What happens then when we move beyond such weights? It would be natural to look at so-called ``large" Bergman spaces, where the weights are rapidly decreasing. One major difficulty  when studying such spaces arises from the lack of an explicit expression for the reproducing kernels  (see \cite{Largebergmanbackground} and the references therein for more information). For a particular example, let $\mathcal{H}_{k}$ denote the Bergman space on $\mathbb{D}$ with weight  $\exp\big(-\frac{1}{1-|z|^2}\big)$ and let 
 $$z_j=1-2^{-j}, \hspace{0.2 cm} w_j=z_j+i2^{-5j/4}, \text{ }j\ge 0.$$
As explained in \cite[Example 4.13]{InterpolatingPick}, $\{z_j\}\cup\{w_j\}$ forms a sequence in $\mathbb{D}$ that is $\mathcal{H}_{k}$-interpolating (in particular, $\{z_j\}\cup\{w_j\}$ must be weakly separated by $k$), but not weakly separated by $H^{\infty}=\text{Mult}(\mathcal{H}_{k})$. Thus, $k$ does not have the multiplier separation property and Proposition \ref{101} does not apply.
\par

Still, it turns out that $k$, as well as the kernels of many other large Bergman spaces on $\mathbb{D}$, is an AS kernel. 
\begin{exmp}[\textit{Large Bergman spaces on $\DD$}]\label{1002}  \hfill  \par
Given an increasing function $h: [0, 1)\to [0, \infty)$ such that $h(0)=0$ and $\lim_{r\to 1}h(r)=+\infty$, we extend it by $h(z)=h(|z|), z\in\DD.$ We also assume that $h\in C^2(\DD)$ and $\Delta h(z)=\Big(\frac{\partial^2}{\partial x^2}+\frac{\partial^2}{\partial y^2} \Big)h(x+iy)\ge 1.$ Define the weighted Bergman space
$$A^2_{h}(\DD)=\bigg\{f\in\text{Hol}(\DD): ||f||^2_{h}=\int_{\DD}|f(z)|^2e^{-2h(z)}dm(z)<\infty\bigg\},$$
where $dm$ is area measure.\par  $A^2_{h}(\DD)$-interpolating sequences for weights of polynomial growth (i.e. $h(z)=-\alpha\log(1-|z|), \text{ } \alpha>0$) were characterized by Seip in \cite{Seipinventiones} (see also \cite{Seip} for a thorough account). Later on, Borichev, Dhuez and Kellay \cite{Largebergman} tackled the case of radial weights of arbitrary (more than polynomial) growth. To state their results, let 
$$\rho(r)=\big[ (\Delta h)(r)\big]^{-1/2}, \hspace{0.3 cm} 0\le r<1.$$
Certain natural growth restrictions are imposed on $\rho$ in the setting of \cite{Largebergman}. Examples of admissible $h$ include 
$$h(r)=\log\log\frac{1}{1-r}\cdot\log\frac{1}{1-r}, \hspace{0.4 cm} h(r)=\frac{1}{1-r} \hspace{0.2 cm} \text{ and } \hspace{0.2 cm}  h(r)=\exp\frac{1}{1-r}.$$
Now, let $\mathcal{D}(z ,r)$ denote the disc of radius $r$ centered at $z$ and define $$d_{\rho}(z, w)=\frac{|z-w|}{\min\{\rho(z), \rho(w) \}}, \hspace{0.3 cm} z,w\in\DD.$$  A subset $\Gamma\subset\DD$ will be called \textit{$d_{\rho}$-separated} if there exists $c>0$ such that $d_{\rho}(z, w)>c,$ for all $z, w\in\Gamma$ such that $z\neq w.$ Also, define the \textit{upper $d_{\rho}$-density} of $\Gamma$ to be 
$$D^{+}_{\rho}(\Gamma)=\limsup_{R\to\infty}\limsup_{|z|\to 1, z\in\DD}\frac{\text{Card}(\Gamma\cap \mathcal{D}(z, R\rho(z)))}{R^2}.$$
By \cite[Theorem 2.4]{Largebergman}, a sequence $\{\lambda_i\}$ is interpolating for $A^2_{h}(\DD)$ if and only if it is $d_{\rho}$-separated and satisfies $D^{+}_{\rho}(\Gamma)<\frac{1}{2}.$ \par 
 Can Theorem \ref{833} be applied here? Let $\ell_{h}$ denote the associated reproducing kernel. First, note that $A^2_{h}(\DD)$ satisfies condition (Q0), as it contains all polynomials (thus, if $\mu_1, \dots, \mu_m\in\DD$ are $m$ distinct points, we can always find $f\in A^2_{h}(\DD)$ such that $f(\mu_1)=\dots=f(\mu_{m-1})=0$ and $f(\mu_m)\neq 0$). It also satisfies (Q1)-(Q3) and $||(\ell_{h})_{\lambda_i}||\to\infty$ if and only if $|\lambda_i|\to 1$ (see Theorems 3.2-3.3 in \cite{Largebergmanbackground}). Next, we observe that the conditions of Lemma \ref{930} are satisfied as well. Indeed, if $\{\mu_i\}\subset\DD$ converges to a point in $\DD$ and $\{\nu_i\}\subset\DD$ satisfies $||(\ell_h)_{\nu_i}||\to\infty$, then the fact that $\rho$ decreases to $0$ near the point $1$ (one of the additional restrictions imposed on $\rho$) implies that we can find $\epsilon>0$ such that $d_{\rho}(\mu_i, \nu_i)>\epsilon$ for all $i$. Thus, we can deduce that $\{\lambda_i\}$ is interpolating for $A^2_{h}(\DD)$ if and only if it is weakly separated\footnote{Actually, in view of Remark \ref{935}, one can easily show that $d_{\rho}$-separation is equivalent to weak separation by $\ell_h$.} by $\ell_h$ and satisfies $D^{+}_{\rho}(\Gamma)<\frac{1}{2}.$ Finally, suppose that $\{\lambda_i\}$ is a weakly separated by $\ell_h$ sequence that can be written as $\{\lambda_i\}=\cup^{\infty}_{j=1}\cup^n_{k=1}\{\mu^k_j\}$, where $\big|\big|(\ell_h)_{\mu^k_j}\big|\big|\to\infty$ as $j\to\infty,$ for every $1\le k\le n.$ Then, we can always find a subsequence $\{m_j\}$ such that  $S=\cup^{\infty}_{j=1}\cup^n_{k=1}\{\mu^k_{m_j}\}$ satisfies $D^{+}_{\rho}(S)=0$. Hence, $S$ must be interpolating for $A^2_{h}(\DD)$ and Theorem \ref{833} allows us to conclude that $\ell_h$ has the AS property.

\end{exmp}

For our next example, we turn to spaces of Bargmann-Fock type. Since these are spaces of entire functions, their multiplier algebras will consist solely of constants. This implies that the associated reproducing kernels will not contain any (non-trivial) complete Pick factors, thus Bargmann-Fock spaces are not really relevant in the context of Question \ref{4}. Still, it is worth noting that their kernels satisfy, in general, the AS property.

\begin{exmp}[\textit{Bargmann-Fock spaces on $\mathbb{C}^n$}]\label{1003}
Given $n\ge 1$ and $\alpha>0,$ the Bargmann-Fock space over $\mathbb{C}^n$ is defined as
$$F^2_{\alpha}=\bigg\{f\in\text{Hol}(\mathbb{C}^n): ||f||^2_{\alpha}=\int_{\mathbb{C}^n}|f(z)|^2e^{-\alpha|z|^2}dm(z)<\infty\bigg\}, $$
where $dm$ denotes Lebesgue measure on $\mathbb{C}^n$. \par 
Interpolating sequences for $F^2_{\alpha}$ in one variable have been completely characterized by Seip \cite{SeipFock} and Seip-Wallst\'{e}n \cite{SeipWallstFock}. For the general case, a  sufficient condition for $F^2_{\alpha}$-interpolation was given by Massaneda and Thomas in \cite{MassanedaPascal} (the authors also gave a necessary condition, although, as they admit, the gap between the two is rather large). Given $\{\lambda_i\}\subset\mathbb{C}^n$, we say that $\{\lambda_i\}$ is \textit{separated} if there exists $\delta>0$ such that $|\lambda_i-\lambda_j|>\delta$, for all $i\neq j.$ Also, let $B(z, r)$ denote the ball of center $z\in\mathbb{C}^n$ and radius $r.$ Given $\Gamma\subset\mathbb{C}^n,$ the \textit{upper density} of $\Gamma$ is defined as 
$$D^{+}(\Gamma)=\limsup_{r\to\infty}\sup_{z\in\mathbb{C}^n}\frac{\text{Card}(\Gamma\cap B(z,r))}{r^2}.$$
\cite[Theorem 5.1]{MassanedaPascal} states that any    $\{\lambda_i\}\subset\mathbb{C}^n$ that is separated and satisfies $D^{+}(\{\lambda_i\})<\alpha/n$ must be interpolating for  $F^2_{\alpha}$.  \par 
Now, let $\ell(z, w)=e^{\alpha z\overline{w}}$ denote the reproducing kernel of $F^2_{\alpha}$. Since $|\langle \hat{\ell}_{z}, \hat{\ell}_{w} \rangle|^2=e^{-\alpha|z-w|^2}$, a sequence $\{\lambda_i\}$ is weakly separated by $\ell$ if and only if it is separated. Also, it can be easily verified that $\ell$ satisfies conditions (Q0)-(Q3) from Theorem  \ref{833}. Next, suppose that $\{\lambda_i\}$ is a weakly separated by $\ell$ sequence that can be written as $\{\lambda_i\}=\cup^{\infty}_{j=1}\cup^n_{k=1}\{\mu^k_j\}$, where $\big|\big|\ell_{\mu^k_j}\big|\big|\to\infty$ as $j\to\infty,$ for every $1\le k\le n.$ This implies that $|\mu^k_j|\to\infty$, for every $1\le k\le n,$ hence we can always find a subsequence $\{m_j\}$ such that  $S=\cup^{\infty}_{j=1}\cup^n_{k=1}\{\mu^k_{m_j}\}$ satisfies $D^{+}(S)=0$. In view of \cite[Theorem 5.1]{MassanedaPascal}, $S$ must be interpolating for $F^2_{\alpha}$. By Theorem \ref{833}, we can conclude that $\ell$ has the AS property.
\end{exmp}
 \begin{remark}
 Let $\phi:\mathbb{C}\to\mathbb{R}$ be a subharmonic function and consider the weighted Bargmann-Fock space defined on $\DD$ by
 $$F^2_{\phi}=\bigg\{f\in\text{Hol}(\mathbb{C}): ||f||^2_{\phi}=\int_{\mathbb{C}}|f(z)|^2e^{-2\phi(z)}dm(z)<\infty\bigg\}.$$
 In \cite{BerndtssonCerdaSufficiency} and \cite{CerdaSeipNecessity}, the results of Seip \cite{SeipFock} and Seip-Wallst\'{e}n \cite{SeipWallstFock} were extended to $F^2_{\phi}$ with $\Delta \phi 	\simeq 1$. Later on, Marco, Massaneda and Ortega-Cerd\`a \cite{MarcoMassanedaCerda} described interpolating sequences for $F^2_{\phi}$ for a wide class of $\phi$ such that $\Delta\phi$ is a doubling measure. A further extension was achieved by Borichev, Dhuez and Kellay in \cite{Largebergman}, where a class of radial $\phi$ having more than polynomial growth was considered. Somewhat more recently, a sufficient condition for interpolation in ``small" Bargmann-Fock spaces (where $\phi(z)=\alpha(\log^{+}|z|)^2)$ was given by Baranov, Dumont, Hartmann and Kellay in \cite[Theorem 1.6]{SmallFock}. A common characteristic shared by all conditions that appear in the previously mentioned results (regardless of whether they are both necessary and sufficient or merely sufficient for $F^2_{\phi}$-interpolation) is that they have the form: $\{\text{\{separation by }d$\}$ \text{ }+ \text{ }\text{(D)}\}$, where $d: \mathbb{C}\times \mathbb{C}\to[0,\infty)$ plays the role of a distance function and (D) is a density condition that is, roughly, always satisfied by any finite union of ``sufficiently sparse" sequences. Thus, letting $\ell_{\phi}$ denote the kernel of $F^2_{\phi}$, Theorem \ref{833} tells us that, for any $\phi$ corresponding to one of the previous cases and such that $\ell_{\phi}$ satisfies (Q0)-(Q3) and $d$ satisfies the hypotheses of Lemma \ref{930},  $\ell_{\phi}$ will have the AS property.  
  
 \end{remark}
 We end this subsection by giving a general class of pairs $(s, \ell)$ for which Question \ref{4} has a positive answer. 
 \begin{theorem}\label{1000}
Let $\ell$ be a kernel on $X$ defined by performing a finite number of any of the operations from Proposition \ref{942} to one or more kernels having the $2$-point Pick property and/or to one or more kernels from Example \ref{1002}. Suppose also that $s$ is a complete Pick factor of $\ell$. Then, a sequence $\{\lambda_i\}\subset X$
is interpolating for $\text{Mult}(\mathcal{H}_{s}, \mathcal{H}_{\ell})$ if and only if it satisfies the Carleson measure condition for $\mathcal{H}_s$ and is weakly separated by $\ell$.

 \end{theorem}
 \begin{proof}
This is a consequence of Corollary \ref{507} and of Propositions \ref{101} and \ref{942}, since $2$-point Pick kernels and the kernels from Example \ref{1002} are all AS kernels. 
 \end{proof}

\subsection{A holomorphic non-example} \label{304}
 In this subsection, we construct a holomorphic pair $(s, \ell)$ on $\DD^2$, where $s$ is a CP factor of $\ell$ and $\ell$ satisfies properties (Q0)-(Q3) from Theorem \ref{833}, such that there exists an (infinite) sequence $\{\lambda_i\}\subset\DD^2$ satisfying the Carleson measure condition for $\mathcal{H}_{s}$ and being weakly separated by $\ell$, but not $\text{Mult}(\mathcal{H}_{s}, \mathcal{H}_{\ell})$-interpolating. In particular, we will construct a sequence that is weakly but not $4$-weakly separated by $\ell.$
 \begin{exmp}\label{1100}
 Let $\mathcal{H}_k$ denote the weighted Bergman space on $\DD$ with weight $e^{-\frac{1}{1-|z|^2}}$ and define the kernels
 $$\ell((\lambda_1, \lambda_2),(\mu_1, \mu_2))=\frac{k(\lambda_1, \mu_1)+k(\lambda_2, \mu_2)}{(1-\lambda_1\overline{\mu}_1)(1-\lambda_2\overline{\mu}_2)},$$
 and 
$$s((\lambda_1, \lambda_2),(\mu_1, \mu_2))=\frac{1}{2-\lambda_1\overline{\mu}_1-\lambda_2\overline{\mu}_2}=\frac{\frac{1}{2}}{1-\big\langle \frac{1}{\sqrt{2}}(\lambda_1, \lambda_2), \frac{1}{\sqrt{2}}(\mu_1, \mu_2)\big\rangle}, $$
 where $(\lambda_1, \lambda_2),(\mu_1, \mu_2)\in\DD\times\DD.$ \par Evidently, $s$ is a complete Pick kernel. Also, letting $H^2_{\DD^2}$ denote the Hardy space on $\DD^2,$ we observe that the vector-valued function $\phi(\lambda_1, \lambda_2)= \begin{bmatrix}
    \lambda_1/\sqrt{2} &  \lambda_2/\sqrt{2} 
 \end{bmatrix}\in\text{Mult}(H^2_{\DD^2}\otimes \mathbb{C}^2, H^2_{\DD^2})$ is a contractive multiplier. As seen in subsection \ref{1030}, this implies that $s$ is a complete Pick factor of $1/(1-\lambda_1\overline{\mu}_1)(1-\lambda_2\overline{\mu}_2)$ and hence also of $\ell.$ Regarding $\ell$, it is easily verified that $\ell((\lambda_1, \lambda_2),(\mu_1, \mu_2)) $ is of the form $\sum^{\infty}_{n, m=0}a_{n, m}(\lambda_1\overline{\mu_1})^n(\lambda_2\overline{\mu_2})^m$, where every $a_{n, m}$ is nonzero. We deduce that the monomials $\lambda_1^n\lambda_2^m$ form a complete orthogonal set for $\mathcal{H}_{\ell}$. Thus, $\ell$ satisfies (Q0) (and the same must be true for $s$). In view of \cite[Theorems 3.2-3.3]{Largebergmanbackground}, $\ell$ must also satisfy (Q1)-(Q3) (this is because $k$ already satisfies these properties). Finally, it is worth noting that the existence of the factor $1/(1-\lambda_1\overline{\mu}_1)(1-\lambda_2\overline{\mu}_2)$ implies that Mult$(\mathcal{H}_{\ell})=H^{\infty}(\DD^2)$. \par 
Now, define the sequence 
$$\lambda_{4j+n}=(z_{4j+n}, w_{4j+n})=\begin{cases} 
(1-2^{-j}, 1-2^{-j}), & \text{ } \text{ if } n=0;
\\ \\
(1-2^{-j}, 1-2^{-j}+i2^{-5j/4}), & \text{ } \text{ if } n=1;
\\ \\
(1-2^{-j}+i2^{-5j/4}, 1-2^{-j}), & \text{ } \text{ if } n=2;
\\ \\
(1-2^{-j}+i2^{-5j/4}, 1-2^{-j}+i2^{-5j/4}), & \text{ } \text{ if } n=3,
\end{cases}$$
  where $j\ge 1.$ Letting 
$$K((\lambda_1, \lambda_2),(\mu_1, \mu_2))=k(\lambda_1, \mu_1)+k(\lambda_2, \mu_2),$$
it can be easily verified that
$$K_{\lambda_{4j}}-K_{\lambda_{4j+1}}-K_{\lambda_{4j+2}}+K_{\lambda_{4j+3}}=0,$$ 
  for all $j\ge 1.$ Thus, \begin{equation} \label{1200} \det[\langle \hat{K}_{\lambda_{4j+n}}, \hat{K}_{\lambda_{4j+m}}\rangle]_{0\le n, m\le 3}=0, \end{equation}
  for all $j\ge 1.$ Also, let $r$ and $t$ denote the Szeg{\H o} kernels on $\DD$ and $\DD^2,$ respectively. A short calculation reveals that $$\lim_j\frac{r(1-2^{-j}, 1-2^{-j}+i2^{-5j/4})}{\sqrt{r(1-2^{-j}, 1-2^{-j})r(1-2^{-j}+i2^{-5j/4}, 1-2^{-j}+i2^{-5j/4})}}=1.$$ 
  Thus, we obtain 
  $$\lim_j\langle \hat{t}_{\lambda_{4j+n}}, \hat{t}_{\lambda_{4j+m}}\rangle =\lim_j\big[\langle \hat{r}_{z_{4j+n}}, \hat{r}_{z_{4j+m}}\rangle \cdot
  \langle \hat{r}_{w_{4j+n}}, \hat{r}_{w_{4j+m}}\rangle \big]=1,$$
  for all $n, m\in\{0, 1, 2, 3\}.$
  In view of (\ref{1200}) and the previous limit, we can write 
   $$ \det[\langle \hat{\ell}_{\lambda_{4j+n}}, \hat{\ell}_{\lambda_{4j+m}}\rangle]_{0\le n, m\le 3}$$
      \begin{equation} \label{1209}=\det[\langle \hat{K}_{\lambda_{4j+n}}, \hat{K}_{\lambda_{4j+m}}\rangle\cdot \langle \hat{t}_{\lambda_{4j+n}}, \hat{t}_{\lambda_{4j+m}}\rangle]_{0\le n, m\le 3}\to 0,
    \end{equation}
  as $j\to\infty.$ \\ Next, we prove the existence of $\epsilon>0$ such that 
  \begin{equation}\label{1210}
   d_{\ell}(\lambda_{4j+n}, \lambda_{4j+m})  >\epsilon,  
  \end{equation}
   $\text{for all }j\ge 1 \text{ and } n, m\in\{0, 1, 2, 3\}, n\neq m$. First, note that the discussion preceding Example \ref{1002} implies the existence of $\delta>0$ such that 
   \begin{equation} \label{1304} d_k(1-2^{-j}, 1-2^{-j}+i2^{-5j/4})>\delta, \hspace{0.3 cm} \forall j\ge 1.\end{equation}
Now, let $j\ge 1 \text{ and } n, m\in\{0, 1, 2, 3\},$ with $n\neq m$. Without loss of generality, we may assume that $z_{4j+n}=1-2^{-j}$ and $z_{4j+m}=1-2^{-j}+i2^{-5j/4}$ (if $z_{4j+n}=z_{4j+m}$, we would work with $w_{4j+n}$ and $w_{4j+m}$ instead). Note that $|w_{4j+n}|, |w_{4j+m}|, |z_{4j+n}|\le |z_{4j+m}| $, hence $||k_{w_{4j+n}}||, ||k_{w_{4j+m}}||, ||k_{z_{4j+n}}|| \le ||k_{z_{4j+m}}||$ ($k$ is rotationally invariant). Also, in view of  (\ref{1303}) and (\ref{1304}), we can find $f\in\mathcal{H}_k$ such that $||f||_{\mathcal{H}_k}<1/\delta$, $f(z_{4j+n})=0$ and $f(z_{4j+m})=||k_{z_{4j+m}}||.$ Define 
$$F(\lambda_1, \lambda_2)=\frac{\sqrt{||k_{z_{4j+m}}||^2+||k_{w_{4j+m}}||^2}}{||k_{z_{4j+m}}||}f(\lambda_1), \hspace{0.3 cm} (\lambda_1, \lambda_2)\in\DD^2.$$
 \cite[Theorem 5.4]{PauRa} implies that $||f||_{\mathcal{H}_K}\le ||f||_{\mathcal{H}_k}$. Thus, $$||F||_{\mathcal{H}_K}\le \frac{\sqrt{||k_{z_{4j+m}}||^2+||k_{z_{4j+m}}||^2}}{||k_{z_{4j+m}}||} ||f||_{\mathcal{H}_k}<\sqrt{2}/\delta.$$
 Also, observe that $F(\lambda_{4j+n})=0$ and 
$$F(\lambda_{4j+m})=\frac{\sqrt{||k_{z_{4j+m}}||^2+||k_{w_{4j+m}}||^2}}{||k_{z_{4j+m}}||}f(z_{4j+m})=||K_{\lambda_{4j+m}}||.$$
In view of (\ref{1303}), we obtain $d_{K}(\lambda_{4j+n}, \lambda_{4j+m})  >\delta/\sqrt{2}$. An application of Lemma \ref{944} then gives us (\ref{1210}). \\
Next, note that both $||\ell_{\lambda_{4j+n}}||, ||s_{\lambda_{4j+n}}||\to\infty$ as $j\to\infty,$ for all $n\in\{0, 1, 2, 3\}.$ Thus, there exists a subsequence $\{m_j\}$ such that $\cup_{n=0}^3\{\lambda_{4m_j+n}\}$ satisfies the Carleson measure condition for $\mathcal{H}_s$. Also, after some calculations, we can deduce the existence of $\epsilon'>0$ with the property
 \begin{equation} 
   d_{t}(\lambda_{4j+n}, \lambda_{4i+m})  >\epsilon',  
  \end{equation}
  for all $i\neq j$ and all $n, m\in\{0, 1, 2, 3\}.$ Lemma \ref{944} then implies that
  \begin{equation}\label{1211}
   d_{\ell}(\lambda_{4j+n}, \lambda_{4i+m})  >\epsilon',  
  \end{equation}
  for all $i\neq j$ and all $n, m\in\{0, 1, 2, 3\}.$ (\ref{1210}) combined with (\ref{1211}) tell us that $\cup_{n=0}^3\{\lambda_{4j+n}\}$ (and hence $\cup_{n=0}^3\{\lambda_{4m_j+n}\}$ as well) is weakly separated by $\ell.$ However, Lemma \ref{708} and (\ref{1209}) imply that $\cup_{n=0}^3\{\lambda_{4j+n}\}$ (and hence $\cup_{n=0}^3\{\lambda_{4m_j+n}\}$ as well) is not $4$-weakly separated by $\ell$. We deduce that the pair $(s, \ell)$ constitutes a counterexample to Question \ref{4}.
 \end{exmp}
\begin{remark}
The specific choice of $s$ in the previous example is not important; all that was required for the proof to go through was a CP factor $s$ of $\ell$ such that $||s_{\lambda_{4j+n}}||\to\infty$ as $j\to\infty,$ for all $n\in\{0, 1, 2, 3\}.$
\end{remark}

 \par\textit{Acknowledgements}.  I would like to thank John M\raise.5ex\hbox{c}Carthy, Michael Hartz and Brett Wick for helpful discussion. I also thank the Onassis Foundation for providing financial support for my PhD studies.

\printbibliography

\end{document}